\documentclass{amsart}
\usepackage{graphicx} 
\usepackage{url}
\usepackage{mathtools}
\usepackage{amsmath,amssymb,amsthm}
\usepackage{relsize}
\usepackage{color}
\usepackage{hyperref}
\usepackage{cleveref}
\usepackage{comment}
\usepackage{enumitem}
\usepackage{braket}
\usepackage{tikz-cd}
\usepackage{stmaryrd}
\usepackage{tablefootnote} 

\newtheorem{theorem}{Theorem}[section]
\newtheorem{lemma}[theorem]{Lemma}
\newtheorem{corollary}[theorem]{Corollary}
\newtheorem{proposition}[theorem]{Proposition}

\theoremstyle{definition}
\newtheorem{definition}[theorem]{Definition}
\newtheorem{example}[theorem]{Example}
\newtheorem*{remark*}{Remark}
\newtheorem{remark}[theorem]{Remark}

\numberwithin{equation}{section} 

\allowdisplaybreaks

\DeclarePairedDelimiter\abs{\lvert}{\rvert}%
\DeclarePairedDelimiter\norm{\lVert}{\rVert}%

\makeatletter
\let\oldabs\abs
\def\abs{\@ifstar{\oldabs}{\oldabs*}}
\let\oldnorm\norm
\def\norm{\@ifstar{\oldnorm}{\oldnorm*}}
\makeatother

\newcommand{\GL}{\operatorname{GL}}
\newcommand{\PGL}{\operatorname{PGL}}

\newcommand{\GSpin}{\operatorname{GSpin}}
\newcommand{\GO}{\operatorname{GO}}

\newcommand{\Lie}{\operatorname{Lie}}
\newcommand{\im}{\operatorname{im}}
\newcommand{\der}{\operatorname{der}}

\newcommand{\gllie}{\mathfrak{gl}}
\newcommand{\sllie}{\mathfrak{sl}}
\newcommand{\splie}{\mathfrak{sp}}
\newcommand{\solie}{\mathfrak{so}}
\newcommand{\rk}{\operatorname{rk}}

\newcommand{\Aut}{\operatorname{Aut}}
\newcommand{\Out}{\operatorname{Out}}
\newcommand{\Sym}{\operatorname{Sym}}

\newcommand{\Hom}{\operatorname{Hom}}
\newcommand{\Std}{\operatorname{Std}}
\newcommand{\Spin}{\operatorname{Spin}}

\newcommand{\id}{\operatorname{id}}
\newcommand{\WD}{\operatorname{WD}}
\newcommand{\rec}{\operatorname{rec}}

\newcommand{\Qlbar}{\ensuremath{\overline{\mathbb{Q}_\ell}}}
\newcommand{\Qlzerobar}{\ensuremath{\overline{\mathbb{Q}_{\ell_0}}}}

\newcommand{\Frob}{\operatorname{Frob}}

\newcommand{\Gal}{\operatorname{Gal}}
\newcommand{\Ind}{\operatorname{Ind}}
\newcommand{\Res}{\operatorname{Res}}
\newcommand{\HT}{\operatorname{HT}}

\newcommand{\eps}{\varepsilon}
\newcommand{\bb}[1]{\mathbb{#1}}
\newcommand{\mc}[1]{\mathcal{#1}}
\newcommand{\mf}[1]{\mathfrak{#1}}
\newcommand{\wt}[1]{\widetilde{#1}}

\newcommand{\ol}[1]{\overline{#1}}

\author{Zachary Feng}
\author{Dmitri Whitmore}

\title[Irreducibility of Galois representations in many dimensions]{Irreducibility of polarized automorphic Galois representations in infinitely many dimensions}

\begin{document}

\begin{abstract}
    Let \( \pi \) be a polarized, regular algebraic, cuspidal automorphic representation of \( \GL_n(\bb{A}_F) \) where \( F \) is totally real or imaginary CM, and let \( (\rho_\lambda)_\lambda \) be its associated compatible system of Galois representations.
    Suppose that \( 7\nmid n \) and, if \( 4\mid n \), then $n = 4p$ for some prime number $p$.
    We prove that there is a Dirichlet density \( 1 \) set of rational primes \( \mc{L} \) such that whenever \( \lambda\mid \ell \) for some \( \ell\in \mc{L} \), then \( \rho_\lambda \) is irreducible.
\end{abstract}

\maketitle
\tableofcontents

\section{Introduction}

\subsection{Overview and main result}

It is a prediction of the Langlands program that if we can attach a \(\ell\)-adic Galois representation $\rho$ to a cuspidal automorphic representation \( \pi \) of \( \GL_n(\bb{A}_F) \) for a number field \( F \), then $\rho$ must be irreducible. For a general number field \( F \), we do not know how to construct their associated Galois representations, let alone determine their irreducibility. However, if \( F \) is totally real or imaginary CM, and \( \pi \) is polarizable and regular algebraic, then there is a well-known geometric construction due to the work of many authors, including but not limited to: Bellaiche, Caraiani, Chenevier, Clozel, Harris, Kottwitz, Labesse, Shin, and Taylor (see \cite[Theorem 2.1.1]{BLGGT14} for a concise reference). Without the polarized assumption, two constructions are known due to \cite{HLTT16} and \cite{Scholz15}.
Letting \( (\rho_\lambda) \) be the compatible system of Galois representations attached to \( \pi \), where \( \rho_\lambda:G_F\to\GL_n(\ol{M_{\pi,\lambda}}) \) and \( M_\pi \) is a number field, it is then of interest to show that these Galois representations are irreducible. 

Let \( P = \{\ell\in\bb{N}\mid\ell\text{ prime} \} \). Results on irreducibility typically demonstrate the existence of a subset \( \mc{L} \) of rational primes (possibly depending on $\pi$), where \( \mc{L} \) satisfies one of the following properties, such that whenever \( \lambda\mid\ell\) for some \(\ell\in\mc{L} \), then \( \rho_\lambda \) is irreducible.
\begin{enumerate}
    \item \( \mc{L}\in\mc{P}_{>0} = \{ \mc{L}\subset P \mid \mc{L} \text{ has positive Dirichlet density} \} \)
    \item \( \mc{L}\in\mc{P}_{1} = \{ \mc{L}\subset P\mid \mc{L}\text{ has Dirichlet density }1 \} \)
    \item \( \mc{L}\in\mc{P}_{aa} = \{ \mc{L}\subset P\mid P\setminus\mc{L} \text{ is finite} \} \)
    \item \( \mc{L}\in\mc{P}_{a} = \{ P \} \)
\end{enumerate}
In the polarized and regular algebraic setting, we summarize a few of the known results in the list below.

\begin{table}[htbp]
    \centering
    \begin{tabular}{c|c|c|c}
         \( n \) & \( F \) & \( \mc{L} \) & Reference \\
         \hline
         \( 2 \) & \( \bb{Q} \) & \( \mc{P}_a \) & \cite{Ribet77} \\
         \hline
         \( 2 \) & totally real & \( \mc{P}_a \) & \cite{Taylor95} \tablefootnote{This was first shown by Ribet in a 1984 unpublished letter to Carayol}\\
         \hline
         \( 3 \) & totally real & \( \mc{P}_a \) & \cite{BlaRog92} \\
         \hline
         all & totally real or imaginary CM \tablefootnote{ under the assumption that the weight is \emph{extremely regular}} & \( \mc{P}_1  \) & \cite{BLGGT14} \\
         \hline
        \( \leq 5 \) & totally real & \( \mc{P}_{1} \) & \cite{CalGee13} \\
         \hline
         all & totally real or imaginary CM & \( \mc{P}_{>0} \) & \cite{PatTay15} \\
         \hline
         \( \leq 6 \) & totally real or imaginary CM & \( \mc{P}_1 \) & \cite{Xia19} \\
         \hline
         \( \leq 6 \) & totally real or imaginary CM & \( \mc{P}_{aa} \) & \cite{Hui23}
    \end{tabular}
    \caption{Results on irreducibility in the polarized setting}
    \label{tab:polarized irreducibility results}
\end{table}

There is also recent work in the non-polarized setting.

\begin{table}[htbp]
    \centering
    \begin{tabular}{c|c|c|c}
         \( n \) & \( F \) & \( \mc{L} \) & Reference \\
         \hline
         \( 3 \) & totally real & \( \mc{P}_a \) & \cite{BocHui24}
    \end{tabular}
    \caption{Results on irreducibility in the non-polarized setting}
    \label{tab:non-polarized irreducibility results}
\end{table}

Returning to the polarized setting, although the result of \cite{BLGGT14} is the first to give an irreducibility criteria for general \( n \), it is important to note that the ``extremely regular weight" hypothesis of \cite{BLGGT14} is often not satisfied in practice. Nevertheless, the strategy of \cite{BLGGT14} -- which combines a powerful potential automorphy theorem, results of \cite{LarPin92}, and an analytic argument involving \( L \)-functions -- is integral to subsequent approaches to this problem appearing in \cite{PatTay15,Xia19,Hui23}. However, these successive attempts to weaken the extremely regular weight assumption (to just regular) have incurred significant trade-offs in one way or another, by restricting to either \( n\leq 6 \) or \( \mc{L}\in\mc{P}_{>0} \).

The main innovation of this paper is that we replace the extremely regular weight assumption with regular, while maintaining \( \mc{L}\in\mc{P}_1 \) under only a mild divisibility condition on \( n \) (we allow any \( n \) such that \( 7\nmid n \) and \( 4\nmid n \)). In particular, we prove the following theorem.

\begin{theorem} \label{thm: intro main theorem}
    Let \( F \) be totally real or imaginary CM. Let \( (\pi,\chi) \) be a polarized regular algebraic cuspidal automorphic representation of \( \GL_n(\bb{A}_F) \) and \( (\rho_\lambda)_\lambda \) be its associated compatible system of Galois representations. Suppose that \( 7\nmid n \) and, if \( 4\mid n \), then $n = 4p$ for some prime number $p$.
    Then there exists a Dirichlet density \( 1 \) set of rational primes \( \mc{L} \) such that if \( \ell \in \mc{L} \) and \( \lambda\mid\ell\), then \( \rho_\lambda \) is irreducible.
\end{theorem}

As a consequence of Theorem \ref{thm: intro main theorem} and a density one result for residual irreducibility of summands of regular, weakly compatible systems proved in \cite{BLGGT14}, we obtain the following corollary.

\begin{corollary} \label{cor: intro residual irreducibility}
Under the same assumptions as in Theorem \ref{thm: intro main theorem}, there exists a Dirichlet density \( 1 \) set of rational primes \( \overline{\mc{L}} \) such that if \( \ell \in \overline{\mc{L}} \) and \( \lambda\mid\ell\), then the residual representation \( \overline{\rho_\lambda}|_{G_{F(\zeta_\ell)}} \) is irreducible.
\end{corollary}

\subsection{Summary of our methods}

We now explain our overall strategy. Let \( \pi \) be a polarized regular algebraic cuspidal automorphic representation of \( \GL_n(\bb{A}_F) \) where \( F \) is either totally real or imaginary CM. Let \( (\rho_\lambda)_\lambda \) denote its associated compatible system of Galois representations. A reduction step of Xia (cf. \Cref{section: Xia reduction}) tells us that there exists \( m\mid n \) and a polarized regular algebraic cuspidal automorphic representation \( \pi_1 \) of \( \GL_m(\bb{A}_{F_1}) \) (where \( F_1 \) is imaginary CM) such that if \( (\rho_{1,\lambda})_\lambda \) is the compatible system associated to \( \pi_1 \), then the absolute irreducibility of \( \rho_{1,\lambda} \) and \( \rho_\lambda \) are equivalent, and moreover at a positive Dirichlet density set of primes \( \mc{L}' \) (which is a subset of the positive Dirichlet density subset of primes given by \cite{PatTay15}) in fact \( \rho_{1,\lambda} \) is strongly irreducible. There is also a finite Galois extension \( E_1/F_1^+ \) (independent of \( \lambda \)) such that \( G_{\rho_{1,\lambda}}|_{G_{E_1}} \) is connected (\( G_{\rho} \) being the Zariski closure of the image of a representation \( \rho \)) and such that \( F_1/F_1^+ \) is the maximal CM subextension of \( E_1/F_1^+ \). We can therefore replace \( \pi \) with \( \pi_1 \) and
suppose that \( F/F^+ \) is maximal CM subextension of \( E/F^+ \), where \( E/F \) is a finite extension such that \( G_{\rho_\lambda}|_{G_E} \) is connected.

The next step, carried out in Section \ref{section: tensor decomposition}, is to show that at a place \( \lambda_0 \) dividing a prime in \( \mc{L}' \), there exists a tensor product decomposition \( \rho_{\lambda_0}\cong\otimes_{i=1}^k \rho_i \), where \( k \) is the number of simple factors of \( \Lie(G_{\rho_{\lambda_0}}^{\circ,\der}) \) and \( \mf{g}_i=\Lie(G_{\rho_i}^{\circ,\der}) \) is simple for each \( i=1,\dots,k \). We give an overview of how this is done. First, we show that there is a decomposition \( \rho_{\lambda_0}|_{G_{E}} \cong \otimes_i \rho_i' \) (recalling that \( G_{\rho_{\lambda_0}|_{G_E}} \) is connected) which is unique up to reordering and twisting by characters, and such that each \( \Lie G_{\rho_i'}^{\circ,\der} \) is simple. 
Importantly, we can take each $\rho_i'$ to be geometric, using results of \cite{patrikis2019variations}.
This gives us a diagram
\[ \begin{tikzcd}
    \prod_{i=1}^k \GL_{n_i}(\ol{\bb{Q}}_l)\ar{r}{\kappa} & \GL_n(\ol{\bb{Q}}_l) \\
    G_E\ar{u}{\prod_i\rho_i'}\ar[swap]{ru}{\rho_{\lambda_0}|_{G_E}} &
\end{tikzcd}
\]
where \( \kappa \) is the Kronecker product of matrices. So we have shown that \( \rho_{\lambda_0}(G_E) \subset \im(\kappa) \). Our goal now is to show that actually \( \rho_{\lambda_0}(G_F)\subset \im(\kappa) \), which more or less gives us the claim.

Fix a complex conjugation \( c\in\Gal(\ol{\bb{Q}}/\bb{Q}) \) and let \( [c] \) denote its \( \Gal(\ol{\bb{Q}}/\bb{Q}) \)-conjugacy class. Since \( F/F^+ \) is the maximal CM subextension of \( E/F^+ \), the group \( \Gal(E/F) \) is generated by the set \( S = \{ (\sigma_1\sigma_2)|_E:\sigma_1,\sigma_2\in [c] \}\subset \Gal(E/F^+) \). Therefore, it suffices to show that for all \( \gamma\in S \), any lift \( \wt{\gamma}\in G_F \) of \( \gamma \) satisfies \( \rho_{\lambda_0}(\wt{\gamma})\in \im(\kappa) \).  
By an argument involving Schur's lemma, it further suffices to show that if \( \sigma\in[c] \), then 
\begin{equation} \label{eqn: intro main technical result}
(\rho_i')^\sigma\cong (\rho_i')^\vee\otimes \chi_{\sigma,i}    
\end{equation}  
(for some character \( \chi_{\sigma,i} \)), since for any lift \( \wt{\gamma}\in G_F \) of \( \gamma\in S \), we would have \( (\rho_i')^{\wt{\gamma}}\cong \rho_i'\otimes \varphi_{\gamma,i} \) (for some character \( \varphi_{\gamma,i} \)).

Establishing (\ref{eqn: intro main technical result}) is one of the key technical results of this paper, with our argument inspired by the proofs in \cite[\S7]{Xia19}.
While we do not know a priori whether each $\rho_i$ is pure, we can at least show (by considering the de Rham character $\det \rho_i$) that for almost all \( v \): if $\alpha$ is an eigenvalue of $\rho_i'(\Frob_v)$, then $\alpha^n$ is a $q_v$-Weil number. We apply this observation at a prime $v$ such that $\rho_{\lambda_0}|_{G_E}(\Frob_v)$ has distinct $n$th powers of eigenvalues and accordingly we adjoin certain $n$th powers to the CM field of definition $M_\pi$. This comes at the expense of possibly shrinking $\mc{L}'$ to a smaller (but still positive density) subset, and replacing $\lambda_0$ with another prime lying above a prime in $\mc{L}'$. The condition on the eigenvalues of $\rho_{\lambda_0}|_{G_E}(\Frob_v)$ then forces the conjugate of $\rho_i'$ under the action of (an extension of) complex conjugation on coefficients to be a twist of $(\rho_i')^\vee$. The isomorphism (\ref{eqn: intro main technical result}) now follows as a consequence of regularity of $\rho_{\lambda_0}$ and our choice of $\mc{L}'$.

At this point, we have a tensor product decomposition \( \rho_{\lambda_0}\cong\otimes_{i=1}^k\rho_i \) as above. We can apply potential automorphy to the collection \( \{ \rho_i \}_{i=1}^k \) to obtain a finite imaginary CM extension \( F'/F \) and polarized regular algebraic cupsidal automorphic representations \( \pi_i \) of \( \GL_{n_i}(\bb{A}_{F'}) \), together with their associated compatible systems \( (\rho_{\pi_i,\lambda})_\lambda \) satisfying \( \rho_{\pi_i,\lambda_0}\cong \rho_i|_{G_{F'}} \) (we note that the new compatible systems might be defined over an enlarged coefficient field, but we are ignoring this technicality for now). It is not difficult to show (by Goursat's lemma and the fact that the rank of \( G_{\rho_\lambda}^{\circ,\der} \) is independent of \( \lambda \)) that if for each \( i \), we can show the (strong) irreducibility of \( (\rho_{\pi_i,\lambda})_\lambda \) at a Dirichlet density \( 1 \) set of primes, then the same must be true for \( (\rho_\lambda)_\lambda \). 

We can therefore replace \( \pi \) with any of the \( \pi_i \) going forward, and henceforth assume there exists \( \lambda_0 \) such that \( \mf{g}_{\lambda_0} := \Lie(G_{\rho_{\lambda_0}}^{\circ,\der}) \) is a simple Lie algebra, and \( \rho_{\lambda_0} \) is strongly irreducible. In Section \ref{section: simple irreducibility}, we iterate through the possibilities for simple Lie algebras \( \mf{g}_{\lambda_0} \) on a case-by-case basis. Let \( \mf{g}_\lambda := \Lie(G_{\rho_{\lambda}}^{\circ,\der}) \) for another place \( \lambda \). Let \( t_{\lambda_0}: \mf{g}_{\lambda_0}\hookrightarrow \mf{gl}_n \) and \( t_{\lambda}:\mf{g}_{\lambda}\hookrightarrow \mf{gl}_n \) be their induced faithful representations. By using a combination of techniques, some of which we describe below, we can show that $\rho_\lambda$ (and hence \( t_{\lambda} \)) is irreducible for \( \lambda\mid\ell\in\mc{L} \) in a Dirichlet density \( 1 \) set of rational primes \( \mc{L} \).

The first key tool is the potential automorphy theorem of \cite{BLGGT14}.
With a suitable choice of $\mc{L}$, we can guarantee that every irreducible summand $\sigma$ of $\rho_\lambda$ which is moreover polarizable, is potentially automorphic, and therefore forms a part of a compatible system of Galois representations. 
If, in addition, one finds that polarizability (and hence potential automorphy) is satisfied by every irreducible summand of $\rho_\lambda$, then an argument of \cite{BLGGT14} using $L$-functions implies that $\rho_\lambda$ must be irreducible.
However, it often suffices to find the existence of a single such summand $\sigma$, due to properties of formal characters in compatible systems, which we now recall.

Building on the work of Serre, \cite[Theorem 3.19]{hui2013monodromy} states that \( \mf{g}_{\lambda_0} \) and \( \mf{g}_{\lambda} \) have the same rank, and that the induced representations \( t_{\lambda_0} \) and \( t_{\lambda} \) have the same formal character. Moreover, since \( \rho_{\lambda_0} \) has regular Hodge-Tate weights, the representation \( t_{\lambda_0} \) must be multiplicity-free.
This multiplicity-freeness plays two important roles, the first being that it greatly simplifies the list of possibilities for $\mf{g}_{\lambda_0}$ and $t_{\lambda_0}$ (see Theorem \ref{thm: multiplicity-free irreps}).
The second consequence is that we obtain useful criteria for the existence of polarizable summands of the Galois representation.
It is proven in \cite{Xia19} that the set $W =\{w + cw: w \text{ a weight of } t_\lambda\}$ is, in a suitable sense, independent of $\lambda$ (see \S \ref{subsect: formal chars and complex conj} for a precise statement).
If $0 \in W$, then for every $\lambda$ we can find a weight $w$ for which $cw = -w$.
Multiplicity-freeness implies both that there is a unique summand $\sigma$ of $\rho_\lambda$ admitting $w$ as a weight, and that this summand must be polarizable.

Other important tools that we exploit, which we will recall in \Cref{subsect: type A formal characters}, come from Chun-Yin Hui's thesis \cite{hui2013monodromy}. Hui proves that every semisimple Lie algebra contains an equal rank subalgebra consisting of only type \( A \) simple factors and that if two faithful representations \( \mf{g}\hookrightarrow\mf{gl}_n \) and \( \mf{g}'\hookrightarrow\mf{gl}_n \) have the same formal character, then \( \mf{g} \) and \( \mf{g}' \) must have the same number of type \( A_k \) factors for \( k\geq 9 \) and \( k=6 \). Therefore, if \( \mf{h}_{\lambda_0}\leq\mf{g}_{\lambda_0} \) and \( \mf{h}_\lambda\leq\mf{g}_\lambda \) are choices of type \( A \) equal rank subalgebras, in which case clearly the induced faithful representations \( t_{\lambda_0}:\mf{h}_{\lambda_0}\hookrightarrow \mf{gl}_n \) and \( t_{\lambda}:\mf{h}_{\lambda}\hookrightarrow \mf{gl}_n \) still have the same formal character, then \( \mf{h}_{\lambda_0} \) and \( \mf{h}_\lambda \) must have the same number of type \( A_k \) factors for \( k \) sufficiently large. This is especially useful in the cases when \( \mf{g}_{\lambda_0}=\mf{sl}_k \) for \( k\geq 9 \), for example. When \( t_{\lambda_0} \) and \( t_\lambda \) have the same formal character, we can define inner products on \( \Lambda_{\mf{g}_{\lambda_0}}\otimes\bb{R} \) and \( \Lambda_{\mf{g}_{\lambda}}\otimes\bb{R} \) (where \( \Lambda_{\mf{g}_{\lambda_0}} \) and \( \Lambda_{\mf{g}_{\lambda}} \) are the weight lattices of \( \mf{g}_{\lambda_0} \) and \( \mf{g}_{\lambda} \), respectively) which are moreover isometric. This isometric structure is used to deal with the case where $\mf{g} \cong \sllie_{2k}$.

\subsection{Limitations}

Theorem \ref{thm: intro main theorem} makes two divisibility assumptions on $n$, for separate reasons. 
The assumption that \( 7\nmid n \) is made for the same reason that \cite{Xia19} does not include the possibility that $n=7$.
Specifically, we cannot handle the case where \( \mf{g}_{\lambda_0}\cong \mf{g}_2 \), the exceptional rank \( 2 \) Lie algebra, with \( t_{\lambda_0} \) its unique \( 7 \)-dimensional irreducible representation.
Indeed, \( t_{\lambda_0} \) has the same formal character as the representation of \( \mf{sl}_3 \) acting via \( \Std\oplus\Std^{\vee}\oplus\mathbf{1} \), which is self-dual but does not have self-dual irreducible constituents (as can be seen from restricting $t_{\lambda_0}$ to a copy of $\sllie_3 \subset \mf{g}_2$).
Existing potential automorphy theorems seem to be insufficient to handle the possibility of such non-polarizable summands.

On the other hand, the general case where \( 4\mid n \) is excluded because we can no longer verify oddness of the tensorands \( \rho_i \), a necessary ingredient for applying potential automorphy theorems.
(A similar issue occurs in \cite{CEG22}, where the case of even $n>2$ is consequently excluded from their main theorem.)
It is possible that this hypothesis could be weakened (say to allow \( 4\mid n \) but $8 \nmid n$), if one is willing to perform a similar case analysis in situations where the Lie algebra $\mf{g}_{\lambda_0}$ is not necessarily simple, but at least has a small number of simple factors.
However, some assumption on the power of $2$ dividing $n$ seems to be necessary when applying existing potential automorphy theorems, as Example \ref{example: n=32 fails} shows in the case $n=32$.

One could ask whether the conclusion that $\mathcal{L} \in \mc{P}_1$ could be improved to $\mc{L} \in \mc{P}_{aa}$.
This also appears to be challenging at present, because of our reliance on the density one result of \cite{BLGGT14} showing residual irreducibility of irreducible summands, which seems tricky to avoid in handling certain cases.
This issue was surmountable in \cite{Hui23} since when $n \leq 6$, it is shown that if $\rho_{\lambda}$ is reducible then there must typically exist a summand with type $A$ image, whence the images in characteristics zero and $\ell$ are closely related for almost all $\ell$ (cf. \cite[Theorem 3.12]{Hui23}).
For larger $n$, there are many possibilities for irreducible representations with the same formal character as a reducible representation whose irreducible summands are not of type $A$.

\subsection{Acknowledgements}

We thank James Newton, Jack Sempliner and Jack Thorne for some interesting conversations.
We are grateful to James Newton and Jack Thorne for some helpful comments on an earlier draft of this paper.
We thank Chun-Yin Hui for making us aware of Remark \ref{rmk: hui's comment}. 
We thank Ariel Weiss for helpful comments regarding the history of this problem.
Z.F. acknowledges the support of the Natural Sciences and Engineering Research Council of Canada (NSERC), [ref. no. 577979-2023].
For the purpose of Open Access, the authors have applied a CC BY public copyright licence to any Author Accepted Manuscript (AAM) version arising from this submission.

\subsection{Notation}

We fix the following notation throughout the paper.
Given a perfect field $F$ with algebraic closure $\ol{F}$, we write $G_F$ for the absolute Galois group of $F$.
We fix an algebraic closure $\ol{\bb{Q}}$ of $\bb{Q}$ and, for each prime number $\ell$, an algebraic closure $\Qlbar$ of $\bb{Q}_\ell$ together with an abstract isomorphism $\iota_\ell: \Qlbar \xrightarrow{\sim} \bb{C}$.
We fix an inclusion $j: \ol{\bb{Q}} \hookrightarrow \bb{C}$.
If $F$ is a number field and $v$ a prime of $F$, then we can view the decomposition group at $v$, $G_{F_v}$, as a subgroup of $G_F$ via $\iota_\ell^{-1} \circ j$.
We let $\Frob_v \in G_{F_v} \subset G_F$ denote a lift of the (geometric) Frobenius element.

Let $F$ be a characteristic zero field.
By an ($\ell$-adic) Galois representation of $F$, we mean a continuous homomorphism $\rho: G_F \to \GL_n(\Qlbar)$.
We let $\overline{\rho}: G_F \to \GL_n(\ol{\mathbb{F}}_\ell)$ denote the semisimplification of the mod-$\ell$ reduction of $\rho$.
We let $\epsilon_\ell: G_F \to \Qlbar^\times$ denote the $\ell$-adic cyclotomic character.
If $\sigma \in \Aut(\Qlbar)$, we let ${}^\sigma \rho$ denote the composition of $\rho$ with the automorphism of $\GL_n(\Qlbar)$ induced by $\sigma$.
If $\gamma \in \Aut(\ol{F})$ with $\gamma(F) = F$, let $\rho^\gamma$ denote the composition of the automorphism of $G_F$ induced by $\gamma$ with $\rho$.
More generally, if $F \subset K \subset \ol{F}$ is an intermediate extension and $\ol{\gamma} \in \Aut(K)$ with $\ol{\gamma}(F) = F$, then let $\rho^{\ol{\gamma}} = \rho^{\gamma}$ for some choice of extension $\gamma \in \Aut(\ol{F})$ of $\ol{\gamma}$; the resulting Galois representation is well-defined, up to conjugacy.

Let $F$ be a number field and $\rho: G_F \to \GL_n(\Qlbar)$ a Galois representation which is Hodge--Tate at each $v|\ell$.
For each embedding $\tau: F \hookrightarrow \Qlbar$, let $\HT_\tau(\rho)$ denote the multi-set of $\tau$-Hodge--Tate weights of $\rho$, normalized so that $\HT_\tau(\epsilon_\ell) = -1$.

If \( F_v \) is a \( p \)-adic field, and \( \rho:G_{F_v}\to\GL_n(\Qlbar) \) is a continuous representation (which we assume to be de Rham if the residue characteristic of \( F \) is \(\ell\)), then we write \( \WD(\rho)=(r,N) \) for the associated Weil-Deligne representation, and \( \WD(\rho)^{F-ss} \) for its Frobenius semisimplification. If \( \pi_v \) is an irreducible smooth representation of $\GL_n(F_v)$, we write $\rec_v^T(\pi)$ for the Weil-Deligne representation associated to \( \pi_v \) under the normalized local Langlands correspondence of \cite{harristaylorLLC}.

We call a number field \( F \) a \emph{CM field} if $F$ has an automorphism \( c \) such that for all embeddings \( i:F\hookrightarrow\bb{C} \) one has \( c\circ i = i \circ c \). It follows that either \( F \) is totally real, or \( F \) is a totally imaginary quadratic extension of a totally real field. In the second case, we may also say that \( F \) is an imaginary CM field. In either case, we let \( F^+ \) denote the maximal totally real subfield of \( F \).

\section{Representation-theoretic preliminaries}

\subsection{Weakly compatible systems}

\begin{definition}
Let $K$ and $M$ be number fields.
A weakly compatible system of $n$-dimensional representations of $K$, defined over $M$ and unramified outside of $S$ a finite set of places of $K$, is a collection of Galois representations
\[
(\rho_\lambda: G_K \to \GL_n(\ol{M_\lambda}))_\lambda
\]
indexed by finite places $\lambda$ of $M$ satisfying the following properties.
\begin{enumerate}
    \item For each $\lambda|\ell$, the representation $\rho_\lambda$ is unramified outside of $S \cup \{v | \ell\}$. 
    \item For each $\lambda|\ell$ and $v \not\in S \cup \{v | \ell\}$, the characteristic polynomial
    \[
    Q_{v,\lambda}(X) := \det(1-\rho_\lambda(\Frob_v)X)
    \]
    lies in $M[X]$.
    If additionally $\lambda'|\ell'$ and $v \nmid \ell'$, then we have equalities of characteristic polynomials
    \[
    Q_{v,\lambda}(X) = Q_{v,\lambda'}(X).
    \]
    \item Each $\rho_\lambda$ is de Rham at every $v|\ell$ and crystalline if $v \not\in S$.
    \item For each embedding $\tau: K \hookrightarrow \ol{M}$, the set of $\tau$-Hodge--Tate weights of $\rho_\lambda$ is independent of $\lambda$ and choice of embedding $\ol{M} \hookrightarrow \ol{M_\lambda}$.
\end{enumerate}
\end{definition}

\begin{remark} \label{rmk: operations on compatible systems}
There are many standard algebraic operations that one can perform on weakly compatible systems to obtain new ones.
If $(\rho_\lambda)_\lambda$ and $(\sigma_\lambda)_\lambda$ are weakly compatible systems of $K$, defined over $M$, then the following define weakly compatible systems of $K$, defined over $M$:
\begin{enumerate}
    \item The direct sum $(\rho_\lambda \oplus \sigma_\lambda)_\lambda$.
    \item The tensor product $(\rho_\lambda \otimes \sigma_\lambda)_\lambda$.
    \item The dual $(\rho_\lambda^\vee)_\lambda$.
\end{enumerate}
If $L/K$ is a finite extension then
\begin{enumerate}[resume]
    \item $(\Res_{L}^K \rho_\lambda)_\lambda$ is a weakly compatible system of $L$, defined over $M$.
\end{enumerate}
If $K/N$ is a finite extension then
\begin{enumerate}[resume]
    \item $(\Ind_{K}^N \rho_\lambda)_\lambda$ is a weakly compatible system of $N$, defined over $M$.
\end{enumerate}
If $M'/M$ is a finite extension then
\begin{enumerate}[resume]
    \item $(\rho_{\lambda}: \lambda'|\lambda)_{\lambda'}$ is a weakly compatible system of $K$, defined over $M'$.
\end{enumerate}
\end{remark}

\begin{definition}
\begin{enumerate}
    \item Let $\alpha \in \ol{\bb{Q}}$ be an algebraic number and let $q \in \bb{Z}$ be a prime power. 
We say that $\alpha$ is a $q$-Weil number of weight $w \in \bb{Z}$ if for every embedding $\iota: \overline{\bb{Q}} \hookrightarrow \bb{C}$, we have
\[
|\iota(\alpha)|^2 = q^w.
\]
\item Let $K$ be a number field, and $\rho: G_K \to \GL_n(\Qlbar)$ be a Galois representation unramified outside of a finite set $S$ of primes containing all $\ell$-adic and archimedean primes.
Suppose that for every $v \not\in S$, the characteristic polynomial of $\rho(\Frob_v)$ lies in $\overline{\bb{Q}}[X]$.
We say that $\rho$ is \textit{pure of weight} $w \in \bb{Z}$ if, for every $v \not\in S$ and every eigenvalue $\alpha \in \ol{\bb{Q}}$ of $\rho(\Frob_v)$, $\alpha$ is a $q_v$-Weil number of weight $w$.
Here $q_v$ denotes the size of the residue field at $v$.
\item If $(\rho_\lambda: G_K \to \GL_n(\ol{M_\lambda}))_\lambda$ is a weakly compatible system unramified outside $S$, we say that it is pure of weight $w$ if each $\rho_\lambda$ is pure of weight $w$.
\end{enumerate}
\end{definition}

\begin{lemma} \label{lemma: HT chars are pure}
Let $K$ be a number field and let $\chi: G_K \to \Qlbar^\times$ be a Hodge--Tate (and hence de Rham) character.
Then $\chi$ corresponds to an algebraic Hecke character and is therefore pure of some weight $w \in \bb{Z}$.
\end{lemma}
\begin{proof}
By \cite[Theorem 2.3.13]{patrikis2019variations}, $\chi$ corresponds to some type $A_0$ Hecke character (as in \cite[Definition 2.3.2]{patrikis2019variations}).
Purity now follows; see, for example, \cite[Lemma 1.9]{fargues2006motives}.
\end{proof}

\subsection{Polarized automorphic representations}

Throughout the rest of this section, let \( F \) be a CM field.
\begin{definition}
    A \emph{polarized} regular algebraic cuspidal automorphic representation of \( \GL_n(\bb{A}_F) \) is a pair \( (\pi,\chi) \) such that
    \begin{enumerate}
        \item \( \pi \) is a regular algebraic cuspidal automorphic representation of \( \GL_n(\bb{A}_F) \).
        \item \( \chi: (F^+)^\times\backslash \bb{A}_{F^+}^\times\to\bb{C}^\times \) is a continuous character such that \( \chi_v(-1) \) is independent of \( v\mid \infty \).
        \item \( \pi^c \cong \pi^\vee\otimes (\chi\circ N_{F/F^+}\circ\det) \) where \( \pi^c \) denotes the composition of \( \pi \) with the action of complex conjugation on \( \GL_n(\bb{A}_F) \).
    \end{enumerate}
\end{definition}

\begin{definition}
    A regular algebraic cuspidal automorphic representation \( \pi \) of \( \GL_n(\bb{A}_F) \) is \emph{polarizable} if there is a character \( \chi \) of \( (F^+)^\times\backslash\bb{A}_{F^+}^\times \) such that \( (\pi,\chi) \) is polarized.
\end{definition}

For $\pi$ as above, we will write the finite part $\pi^\infty \cong \otimes_{v}' \pi_v$ as a restricted tensor product of smooth irreducible representations $\pi_v$ of $\GL_n(F_v)$.

\subsection{Polarized Galois representations}

Let \( r:G_F\to\GL_n(\Qlbar) \) and \( \mu: G_{F^+}\to\Qlbar^\times \) be continuous homomorphisms. Let \( [c_v] \) denote the conjugacy class of complex conjugations in \( G_{F^+} \) associated to a place \( v\mid\infty \) of \( F^+ \). (If \( v \) corresponds to the embedding \( i: F\hookrightarrow\bb{C} \), then for each embedding \( j:\ol{F}\hookrightarrow\bb{C} \) extending \( i \), the restricting of the unique complex conjugation on \( \bb{C} \) to \( \ol{F} \) along \( j \) gives rise to an element of \( [c_v] \).)

\begin{definition}
    We say that \( (r,\mu) \) is \emph{polarized} if for some (equivalently every) infinite place \( v \) of \( F^+ \), there exists \( \epsilon_v\in\{\pm 1 \} \) and a non-degenerate pairing \( \braket{\:\:,\:\:}_v \) on \( \Qlbar^n \) such that
    \[ \braket{x,y}_v=\epsilon_v\braket{y,x}_v \]
    and
    \[ \braket{r(\sigma)x,r(c_v\sigma c_v)y}_v = \mu(\sigma)\braket{x,y}_v \]
    for all \( x,y\in\Qlbar \) and \( \sigma\in G_F \). If \( F\neq F^+ \), replacing \( \mu \) by \( \mu\delta_{F/F^+} \) if necessary, we further require that \( \epsilon_v=-\mu(c_v) \) for all \( v \).
    We call $\mu$ the multiplier character.
    
    We say that $r$ is \emph{polarizable} if there exists $\mu$ as above such that $(r,\mu)$ is polarized.
\end{definition}

\begin{definition} \label{defn: totally odd}
    We call a polarized \( \ell \)-adic Galois representation \( (r,\mu) \) \emph{totally odd} if \( \epsilon_v=1 \) for all \( v\mid \infty \) of \( F^+ \).
\end{definition}

Suppose that $(r,\mu)$ and $(s: G_F \to \GL_m(\Qlbar),\chi: G_{F^+} \to \Qlbar^\times)$ are both polarized Galois representations.
Then $(r \otimes s, \mu \chi \delta_{F/F^+})$ is also naturally a polarized Galois representation (\cite[\S 1.1]{BLGGT14}).

\begin{lemma} \label{lemma: oddness}
Suppose that $F$ is an imaginary CM field.
Let $\rho: G_F \to \GL_n(\Qlbar)$ and $\sigma: G_F \to \GL_m(\Qlbar)$ be polarizable, irreducible Galois representations.
\begin{enumerate}
    \item If $n$ is odd, then $\rho$ is totally odd.
    \item If $n = 2$, suppose that the following conditions are satisfied:
    \begin{enumerate}[label=(\roman*)]
        \item $\ell \geq 11$.
        \item $\Sym^2 \rho$ is irreducible and Fontaine--Laffaille at all places dividing $\ell$ (i.e. $\Sym^2 \rho$ satisfies the hypotheses of Lemma \ref{lemma: potentially diagonalizable}).
        \item $\Sym^2 \overline{\rho}|_{G_{F(\zeta_\ell)}}$ is irreducible.
    \end{enumerate}
    Then $\rho$ is totally odd.
        \item If $\rho \otimes \sigma$ is irreducible and two of $\{\rho, \sigma, \rho \otimes \sigma\}$ are totally odd, then all three are totally odd.
\end{enumerate}
\end{lemma}
\begin{proof}
The first part is a special case of \cite[Lemma 1.5.3]{CEG22}.
The second part is \cite[Lemma 3.4.1]{CEG22}.
For the final part, note that the polarization pairings associated to $\rho, \sigma$ and $\rho \otimes \sigma$ are unique (up to scalars) by Schur's lemma.
It follows that the multiplier character on $\rho \otimes \sigma$ is given by the product of the multiplier characters of $\rho$ and $\sigma$ with $\delta_{F/F^+}$.
Since total oddness is equivalent to the multiplier character sending each complex conjugation to $-1$, the result now easily follows. 
\end{proof}

The following theorem is due to the work of many people. We give the result stated in \cite[Theorem 2.1.1]{BLGGT14} as a convenient reference.

\begin{theorem}
    Let \( (\pi,\chi) \) be a polarized, regular algebraic, cuspidal automorphic representation of \( \GL_n(\bb{A}_F) \) of weight \( a\in (\bb{Z}^n)^{\Hom(F,\bb{C}),+} \), i.e. \( a_{\tau,1}\geq\dots\geq a_{\tau,n} \). Then for each prime \( \ell \) and isomorphism \( \iota:\Qlbar\to\bb{C} \), there is a continuous semisimple representation
    \[ r_{\ell,\iota}(\pi):G_F\to\GL_n(\Qlbar) \]
    and an integer \( w \) such that the following properties hold:
    \begin{enumerate}
        \item \( (r_{\ell,\iota}(\pi),\epsilon_\ell^{1-n}r_{\ell,\iota}(\chi)) \) is a totally odd, polarized, \( \ell \)-adic representation.
        \item If \( v\nmid \ell \) is a place of \( F \), then
        \[ \iota\WD(r_{\ell,\iota}(\pi)|_{G_{F_v}})^{F-ss}\cong \rec_v^T(\pi_v) 
        \]
        and these Weil-Deligne representations are pure of weight \( w \).
        \item \( r_{\ell,\iota}(\pi) \) is de Rham, and if \( \tau:F\hookrightarrow\Qlbar \) then
        \[ \HT_{\tau}(r_{\ell,\iota}(\pi)) = \{ a_{\iota\tau,1}+n-1,a_{\iota\tau,2}+n-2,\dots, a_{\iota\tau,n} \}. \]
        Moreover,
        \[ \HT_{\tau\circ c}(r_{\ell,\iota}(\pi)) = \{ w-h\mid h\in \HT_{\tau}(r_{\ell,\iota}(\pi)) \}. \]
    \end{enumerate}
\end{theorem}

\begin{remark} \label{rmk: automorphic Galois reps compatible systems}
    Let \( \pi \) be a polarizable, regular algebraic, cuspidal automorphic representation of \( \GL_n(\bb{A}_{F}) \). 
The collection of Galois representations \( \{ r_{\ell,\iota_\ell}(\pi) \}_{\ell,\iota_\ell} \) gives rise to a weakly compatible system.
Indeed, there exists a number field \( M_\pi\subset\bb{C} \) such that for almost all places \( v \) of \( F \),
    \[ \det(1-\rec_v^T(\pi_v)(\Frob_v)X)\in M_\pi[X]. \]
If $\ell$ is a prime number, there is a canonical decomposition of Galois representations \[r_{\ell,\iota_\ell}(\pi) \otimes_{\bb{Q}} M_\pi \cong \bigoplus_{\tau: M_\pi \hookrightarrow \Qlbar} \rho_{\pi,\tau}.\]
If $\lambda|\ell$ is a place of $M_\pi$, set $\rho_{\pi,\lambda} = \rho_{\pi,\tau}$ for any choice of $\tau: M_\pi \hookrightarrow \Qlbar$ inducing $\lambda$.
The collection $(\rho_{\pi,\lambda})_\lambda$ is a weakly compatible system defined over $M_\pi$.

\end{remark}

By purity, we can take $M_\pi$ to be a CM field (\cite[Lemma 1.2]{PatTay15}).
Moreover, on replacing $M_\pi$ by a further CM extension, we can (and will) assume throughout that every irreducible summand of each member $\rho_{\pi,\lambda}$ of the associated compatible system is defined over $M_{\pi,\lambda}$ (\cite[Lemma 1.4]{PatTay15}).

\begin{definition} \label{defn: automorphic Galois representation}
We say a Galois representation $\rho: G_F \to \GL_n(\Qlbar)$ is automorphic if there exists a polarized, regular algebraic, cuspidal automorphic representation $(\pi,\chi)$ of \( \GL_n(\bb{A}_{F}) \) and an isomorphism $r_{\pi,\iota} \cong \rho$ for some isomorphism $\iota: \Qlbar \to \mathbb{C}$.
\end{definition}

By Remark \ref{rmk: automorphic Galois reps compatible systems}, if $\rho: G_F \to \GL_n(\Qlbar)$ is automorphic, then $\rho$ is part of a weakly compatible system of Galois representations of $F$.

\subsection{Potential automorphy results}

\begin{theorem}[\cite{BLGGT14} Theorem 4.5.1] \label{thm: blggt potential automorphy}
    Suppose that we are in the following situation.
    \begin{enumerate}
        \item Let \( F/F_0 \) be a finite, Galois extension of imaginary CM fields. Let \( F^+ \) and \( F_0^+ \) denote their maximal totally real subfields.
        \item Let \( \mc{I} \) be a finite set.
        \item For each \( i\in\mc{I} \), let \( n_i \) and \( d_i \) be positive integers and \( \ell_i \) be an odd rational prime such that \( \ell_i\geq 2(d_i+1) \) and \( \zeta_{\ell_i}\notin F \). Also choose an isomorphism \( \iota_i: \ol{\bb{Q}_{\ell_i}}\to\bb{C} \) for each \( i\in\mc{I} \).
        \item For each \( i\in\mc{I} \), let \( (r_i,\mu_i) \) be a \textbf{totally odd, regular algebraic}, \( n_i \)-dimensional, \textbf{polarized} \( \ell_i \)-adic representations of \( G_F \) such that \( d_i \) is the maximum dimension of an irreducible constituent of the restriction to \( \ol{r_i} \) to the closed subgroup of \( G_F \) generated by all Sylow pro-\( \ell_i \)-subgroups.
        \item Let \( F^{(avoid)}/F \) be a finite Galois extension.
    \end{enumerate}
    Suppose moreover that the following conditions are satisfied for every \( i\in\mc{I} \).
    \begin{enumerate}
        \item \( r_i \) is potentially diagonalizable (in the sense of \cite[\S 1.4]{BLGGT14})
        at each prime \( v \) of \( F^+ \) above \( \ell_i \).
        \item \( \ol{r_i}|_{G_{F(\zeta_{\ell_i})}} \) is irreducible.
    \end{enumerate}
    Then we can find a finite CM extension \( F'/F \) and for each \( i\in\mc{I} \) a regular algebraic, cuspidal, polarized automorphic representation \( (\pi_i,\chi_i) \) of \( \GL_{n_i}(\bb{A}_{F'}) \) such that
    \begin{enumerate}
        \item \( F'/F_0 \) is Galois.
        \item \( F' \) is linearly disjoint from \( F^{(avoid)} \) over \( F \).
        \item \( \pi_i \) is unramified above \( \ell_i \).
        \item \( (r_{\ell_i,\iota_i}(\pi_i),r_{\ell_i,\iota_i}(\chi_i)\epsilon_{\ell_i}^{1-n_i}) \cong (r_{i}|_{G_{F'}},\mu_i|_{G_{(F')^+}}) \).
    \end{enumerate}
\end{theorem}

\begin{lemma}[{\cite[Lemma 1.4.3]{BLGGT14}}] \label{lemma: potentially diagonalizable}
Suppose that $\ell$ is unramified in $F$.
Let $\rho: G_F \to \GL_n(\Qlbar)$ be a Galois representation which is crystalline at each $v|\ell$.
Suppose that for every embedding $\tau: F \hookrightarrow \Qlbar$ there exists an integer $a_\tau$ such that $\HT_{\tau}(\rho) \subset [a_\tau, a_\tau + \ell - 2]$ (the Fontaine--Laffaille range).
Then $\rho$ is potentially diagonalizable at each $v|\ell$.
\end{lemma}

\begin{proposition}[\cite{BLGGT14} Proposition 5.3.2] \label{proposition: density 1 residual irreducibility of summands}
    Suppose \( \mc{R} \) is a regular, weakly compatible system of \( \ell \)-adic representations of \( G_F \) defined over \( M \). If \( s \) is a subrepresentation of \( r_\lambda \) then we will write \( \ol{s} \) for the semisimplification of the reduction of \( s \). Also write \( \ell \) for the rational prime lying below \( \lambda \). Then there is a Dirichlet density \( 1 \) set of rational primes (depending only on \( \mc{R} \)), such that if \( s \) is any irreducible subrepresentation of \( r_\lambda \) for any \( \lambda \) dividing any element of \( \mc{L} \) then \( \ol{s}|_{G_{F(\zeta_\ell)}} \) is irreducible.
\end{proposition}

The following theorem is a key tool in deducing irreducibility of Galois representations attached to polarizable, regular algebraic, cuspidal automorphic representations, based on the methods of \cite{BLGGT14}.

\begin{theorem} \label{thm: irreducibility from polarized summands}
Let $\pi$ be a polarizable, regular algebraic, cuspidal automorphic representation of \( \GL_n(\bb{A}_F) \), with $F$ an imaginary CM field.
There exists a Dirichlet density $1$ set of primes $\mathcal{L}_\pi$ of $\mathcal{L}$ with the following property.
Let $\ell \in \mathcal{L}_\pi$ and let $\lambda|\ell$ be a prime of $M_\pi$.
Write
\[\rho_{\pi,\lambda} = \bigoplus_{i=1}^{k_\lambda} W_i,\]
with each $W_i$ an irreducible representation.
\begin{enumerate}
    \item If $I \subset \{1,\ldots,k_\lambda\}$ satisfies that $W_i$ is polarizable for each $i \in I$, then there exists a finite CM extension $F'/F$ (independent of \( i\in I \)) such that each $W_i|_{G_{F'}}$ is automorphic.
    \item Moreover, if every $W_i$ is polarizable then $\rho_{\pi,\lambda}$ is irreducible.
\end{enumerate}
\end{theorem}
\begin{proof}
Consider the density one set $\mathcal{L}$ obtained from applying Proposition \ref{proposition: density 1 residual irreducibility of summands} to the weakly compatible system attached to $\pi$.
Define $\mathcal{L}_\pi \subset \mathcal{L}$ to be the cofinite subset of $\ell \in \mc{L}$ such that $\ell \geq 2n+2$ is unramified in $F$, $\pi$ is not ramified at any place $v|\ell$ of $F$ and for which $\rho_{\pi,\lambda}$ has all Hodge--Tate weights in the Fontaine--Laffaille range for $\lambda|\ell$.

Fix $\lambda | \ell \in \mathcal{L}_\pi$.
Each polarized summand $W_i$ is automatically totally odd by the same argument as the proof of \cite[Proposition 4.8]{Hui23}.
The first part now follows from Theorem \ref{thm: blggt potential automorphy}.
The second part then follows from the first, together with \cite[Proposition 4.9]{Hui23}.
\end{proof}

\subsection{Formal characters} \label{subsect: formal characters}

We recall some notions regarding formal (bi-)characters, mostly following \cite[\S 2.3]{Hui23}.

\begin{definition}
Let $\Omega$ be a field and let $G \subset \GL_{n,\Omega}$ be a subgroup.
\begin{enumerate}
    \item Suppose firstly that $\Omega$ is algebraically closed and that $G$ is a connected subgroup with a maximal torus $T \subset G$.
The formal character (resp. bi-character) of $G \subset \GL_n$ is the data of the embedding(s) \[T \subset \GL_n\] (resp. \[T \cap G^{\der} \subset T \subset \GL_n),\] taken up to $\GL_n$-conjugacy.
\item In general, if $\Omega$ is not necessarily algebraically closed and $G$ is not necessarily (geometrically) connected, the formal character of $G \subset \GL_{n,\Omega}$ is the formal character of $(G_{\ol{\Omega}})^\circ \subset \GL_{n,\ol{\Omega}}$.
\item We say that the formal character (resp. bi-character) of $G$ is \emph{multiplicity-free} if the corresponding representation of $T$ (resp. $T \cap G^{\der}$) is isomorphic to a direct sum of pairwise distinct characters.
\end{enumerate}
\end{definition}

Let $\{G_i \subset \GL_{n,\Omega_i}\}_{i \in I}$ be a non-empty collection of subgroups with each $\Omega_i$ a field.
The collection is said to have the same formal character (resp. bi-character) if there exists a diagonal $\bb{Z}$-subtorus \[T \subset \bb{G}_{m,\bb{Z}}^n\] (resp. chain of $\bb{Z}$-subtori \[T' \subset T \subset \bb{G}_{m,\bb{Z}}^n )\] whose base-change to $\ol{\Omega_i}$ gives rise to the formal character (resp. bi-character) of $G_i \subset \GL_{n,\Omega_i}$ for every $i \in I$.

Let $K$ be a number field and let $\rho: G_K \to \GL_n(\ol{M_\lambda})$ be a Galois representation.
Let $G_\rho$ denote the Zariski closure of the image of $\rho$ in \( \GL_{n,\ol{M_\lambda}} \).
We refer to the formal (bi-)character of $\rho$ as the formal (bi-)character $(T' \subset) T \subset \GL_{n,\ol{M_\lambda}}$ of $G_\rho \subset \GL_{n,M_\lambda}$.
We refer to the dimension of $T'$ as the rank of $\rho$.
We let $\mf{g}_\rho = \Lie G_\rho^{\circ, \der}$.

If $(\rho_\lambda: G_K \to \GL_n(\ol{M_\lambda}))_\lambda$ is a weakly compatible system, we will often let $G_\lambda$ (resp. $\mf{g}_\lambda$) denote $G_{\rho_\lambda}$ (resp. $\mf{g}_{\rho_\lambda}$).

\begin{theorem}[{\cite[Theorem 2.4, Remark 2.5]{Hui23}}] \label{thm: formal char independent of l}
Let $(\rho_\lambda: G_K \to \GL_n(\ol{M_\lambda}))_\lambda$ be a weakly compatible system.
\begin{enumerate}
    \item Then the component group $\pi_0(G_{\lambda}) = G_{\lambda}/G_{\lambda}^\circ$ cuts out a finite extension $L/K$ which is independent of $\lambda$.
    \item The formal (bi-)character of $(\rho_\lambda)_\lambda$ is independent of $\lambda$.
    In particular, the rank of $\rho_\lambda$ is independent of $\lambda$.
\end{enumerate}
\end{theorem}

\begin{definition} \label{defn: strongly irreducible}
A Galois representation $\rho: G_K \to \GL_n(\Qlbar)$ is said to be \emph{strongly irreducible} if, for every finite extension $K'/K$, $\rho|_{G_{K'}}$ is irreducible.    
\end{definition}

Note, as in \cite[Remark 3.2.7]{CEG22}, that strong irreducibility of $\rho$ is equivalent to irreducibility of the corresponding representation of $G_{\rho}^\circ$.

\begin{lemma} \label{lemma: formal char is multiplicity-free}
Let $\rho: G_K \to \GL_n(\Qlbar)$ be a Galois representation unramified outside of a finite set $S$ of primes of $K$ and which is Hodge--Tate at each $v|\ell$ with regular Hodge--Tate weights.
\begin{enumerate}
    \item The formal character of $\rho$ is multiplicity-free.
    \item If $\rho$ is also strongly irreducible, then the formal bi-character of $\rho$ is multiplicity-free.
\end{enumerate}
\end{lemma}
\begin{proof}
The first part is shown in the proof of \cite[Lemma 5.3.1]{BLGGT14}.
For the second part, $T$ be a maximal torus of $G_{\rho}^\circ$ and let $T' = T \cap G_{\rho}^{\circ,\der}$.
Then $T = Z(G_\rho^\circ) T'$, since $G_\rho^\circ$ is isogenous to the product of its derived group and its center.
Since $\rho$ is strongly irreducible, $Z(G_\rho^\circ) \subset \GL_{n,\Qlbar}$ is contained inside the group of scalar matrices by Schur's lemma.
Since $T \subset \GL_n$ is multiplicity-free, the same must therefore be true of $T' \subset \GL_n$.
\end{proof}

\subsection{Facts about simple Lie algebras}

In this subsection, let $\mf{h}$ be a simple Lie algebra over an algebraically closed field of characteristic zero.
We say a representation $V$ of $\mf{h}$ is \emph{multiplicity-free} if every weight of $V$ occurs with multiplicity one.

\begin{theorem}[{\cite[Theorem 2.2]{sun2024descriptions}}] \label{thm: multiplicity-free irreps}
Any multiplicity-free non-trivial irreducible representation of $\mf{h}$ is isomorphic to one of the following:
\begin{enumerate}
\item a symmetric power or alternating power of either the degree $m+1$ standard representation or its dual if $\mf{h} \cong \sllie_{m+1}$;
\item the standard representation $V_{\omega_1}$ of degree $2m+1$ or the spin representation $V_{\omega_m}$ of degree $2^m$ if $\mf{h} \cong \solie_{2m+1}, m \geq 2$;

\item the standard representation \( V_{\omega_1} \) of degree $2m$ if $\mf{h} \cong \splie_{2m}$, $m \geq 3$, or the representation \( V_{\omega_3} \) of dimension $14$ when $m=3$;
\item the standard representation \( V_{\omega_1} \) of degree $2m$ or the half-spin representations \( V_{\omega_{m-1}} \), \( V_{\omega_m} \) of degree $2^{m-1}$ if $\mf{h} \cong \solie_{2m}, m \geq 4$;
\item the (non-self-dual) representations \( V_{\omega_1} \), \( V_{\omega_6} \) of $\mf{e}_6$ of dimensions $27$;
\item the representation \( V_{\omega_7} \) of $\mf{e}_7$ of dimension \(56\);
\item the representation \( V_{\omega_1} \) of $\mf{g}_2$ of dimension \(7\).
\end{enumerate}    
\end{theorem}

\begin{lemma} \label{lemma: outer automorphisms simple lie algebras}
The outer automorphism group of $\mf{h}$ is isomorphic to
\begin{itemize}
    \item $\mathbb{Z}/2\bb{Z}$ if $\mf{h}$ is of type $A_n$ $(n \geq 2)$, $D_n$ $(n \geq 5)$ or $E_6$,
    \item $S_3$ if $\mf{h}$ is of type $D_4$, and
    \item trivial otherwise.
\end{itemize}
\end{lemma}
\begin{proof}
The outer automorphism group of $\mf{h}$ is isomorphic to the automorphism group of the Dynkin diagram associated to $\mf{h}$ \cite[Proposition D.40]{fulton2013representation}.
The automorphism groups of Dynkin diagrams of irreducible root systems are stated in \cite[\S 12.2]{humphreys2012introduction}.
\end{proof}

\section{A reduction step of Xia} \label{section: Xia reduction}

Let \( (\pi,\chi) \) be a polarized, regular algebraic, cuspidal automorphic representation of \( \GL_n(\bb{A}_F) \), where \( F \) is a totally real or imaginary CM field.
In this section we reduce the proof of \Cref{thm: intro main theorem} to proving \Cref{theorem: xia's reduction}, which will have the benefit of allowing us to assume the Galois representation $\rho_{\pi,\lambda}$ is strongly irreducible for $\lambda$ lying above a prime in a set of positive Dirichlet density. 

In the case that $F$ is imaginary CM, we follow \cite[\S4]{Xia19} and construct an auxiliary weakly compatible system \( (\Phi_\lambda)_\lambda \) (of \( G_{F^+} \)) where
\[ \Phi_\lambda := (\Ind_{F^+}^{F} \rho_{\pi,\lambda}) \oplus \rho_{\chi,\lambda}. \]
In the case that $F = F^+$ is totally real, we follow \cite[\S4.2.3]{Hui23} and construct an analogous weakly compatible system \( (\Phi_\lambda)_\lambda \) (of \( G_{F^+} \)) where:
\[ \Phi_\lambda := (\Ind_{F^+}^{K} \Res_{F^+}^{K} \rho_{\pi,\lambda}) \oplus \rho_{\chi,\lambda}, \]
with $K/F^+$ an arbitrary quadratic imaginary CM extension. 
In either case, by Theorem \ref{thm: formal char independent of l}, we can find a finite extension \( F_{1,\pi}/F^+ \) such that for all places \( \lambda \) of \( M_\pi \), one has an isomorphism
\[ \Gal(F_{1,\pi}/F^+)\cong G_{\Phi_\lambda}/G_{\Phi_\lambda}^\circ \]
induced via \( \Phi_\lambda \). In particular, \( F_{1,\pi} \) is the minimal extension of \( F^+ \) such that the restriction of \( \Phi_{\lambda} \) to \( F_{1,\pi} \) has connected algebraic monodromy group.

The following two lemmas (stated conveniently in \cite[\S 4]{Hui23}) are essentially contained within the proof of \cite[Proposition 2]{Xia19}. We note that, while the statement of \textit{loc. cit.} assumes irreducibility in a density one subset for certain weakly compatible systems, the argument used to only prove the lemmas does make use of such irreducibility assumptions.

\begin{lemma}[{\cite[Proposition 4.13]{Hui23}}]\label{lemma: work with pi_1 imaginary CM}
    Let \( F \) be an imaginary CM field, and \( (\rho_{\pi,\lambda})_{\lambda} \) be the weakly compatible system of \( G_F \) above.
    Let \( F_2 \) be the maximal CM subextension of \( F_{1,\pi}/F^+ \). After enlarging the CM field \( M_\pi \) if necessary, there exists a family of Galois representations \( (r_{1,\lambda})_\lambda \) of a subextension \( F_4 \) of \( F_2/F \) and a regular algebraic polarizable cuspidal automorphic representation \( \pi_1 \) of \( \GL_m(\bb{A}_{F_3}) \) where \( F_3 \) is a finite CM extension of \( F_2 \) such that
    \[ (\Ind_{F_4}^F r_{1,\lambda})_\lambda \cong (\rho_{\pi,\lambda})_\lambda \qquad\text{and}\qquad (\Res^{F_4}_{F_3} r_{1,\lambda})_\lambda \cong (\rho_{\pi_1,\lambda})_\lambda \]
    and \( F_3 \) is the maximal imaginary CM subextension of \( F_{1,\pi_1}/F^+ \).
\end{lemma}

\begin{lemma}[{\cite[Proposition 4.14]{Hui23}}]\label{lemma: work with pi_1 totally real}
    Let \( F^+ \) be a totally real field and \( (\rho_{\pi,\lambda})_\lambda \) be a weakly compatible system of \(G_{F^+} \) over \( M_\pi \) as above. Fix any imaginary CM field \( K \) containing \( F^+ \) as its maximal totally real subfield. Let \( F_2 \) be the maximal CM subextension of \( F_{1,\pi}/F^+ \). After enlarging the CM field \( M_\pi \) if necessary, there exists a family of Galois representations \( (r_{1,\lambda})_\lambda \) of a subextension \( F_4 \) of \( F_2/F^+ \) and a regular algebraic polarizable cuspidal automorphic representation \( \pi_1 \) of \( \GL_m(\bb{A}_{F_3}) \) where \( F_3 \) is a finite CM extension of \( F_2 \) such that
    \[ (\Ind_{F_4}^{F^+} r_{1,\lambda})_\lambda \cong (\rho_{\pi,\lambda})_\lambda \qquad\text{and}\qquad (\Res^{F_4}_{F_3} r_{1,\lambda})_\lambda \cong (\rho_{\pi_1,\lambda})_\lambda \]
    and \( F_3 \) is the maximal imaginary CM subextension of \( F_{1,\pi_1}/F^+ \).
\end{lemma}

In both \Cref{lemma: work with pi_1 imaginary CM} and \Cref{lemma: work with pi_1 totally real}, we have that the absolute irreducibility of \( \rho_{\pi,\lambda} \) and \( r_{1,\lambda} \) are equivalent and, moreover, follows from the irreducibility of \( \rho_{\pi_1,\lambda} \). This can be argued using Mackey’s irreducibility criterion and the regularity of \( (\rho_{\pi,\lambda})_\lambda \), in the same way as \cite[\S 4.3]{Hui23}.
We have therefore reduced the problem of proving Theorem \ref{thm: intro main theorem} to proving the following theorem.

\begin{theorem}\label{theorem: xia's reduction}
    Let \( F \) be an imaginary CM field. Let \( (\pi,\chi) \) be a regular algebraic polarized cuspidal automorphic representation of \( \GL_n(\bb{A}_F) \) with associated weakly compatible system of Galois representations \( (\rho_{\pi,\lambda})_\lambda \) of \( G_F \) defined over an imaginary CM field \( M_\pi \).
    Suppose that \( 7\nmid n \) and, if \( 4\mid n \), then $n = 4p$ for some prime number $p$.
    Suppose that \( F \) is the maximal CM subextension of \( F_{1,\pi}/F^+ \). Then there exists a Dirichlet density 1 set of primes such that for all places \( \lambda \) dividing a prime in this set, \( \rho_{\pi,\lambda} \) is strongly irreducible.
\end{theorem}

The main advantage of applying this reduction step is that we can assume strong irreducibility in a positive density set of primes.
To define this set, let $N$ be a finite extension of $M_\pi$ containing the images of all embeddings $F_{1,\pi} \hookrightarrow \ol{M_\pi}$.
After replacing $N$ by a further extension if necessary, we may suppose that $N/\bb{Q}$ is Galois.
Choose a complex conjugation $c \in \Gal(N/\bb{Q})$ and let $[c] \subset \Gal(N/\bb{Q})$ denote the conjugacy class of $c$.
We then set
\[\mathcal{L} = \{\ell \mid [\Frob_\ell] = [c], \ell \text{ is unramified in } N \},\]
a set of rational primes of positive Dirichlet density.

\begin{theorem}[{\cite[Corollary 1]{Xia19}}]\label{lemma: lie-irreducible patrikis taylor prime}
   Suppose that $\pi$ satisfies the hypotheses of Theorem $\ref{theorem: xia's reduction}$.
    Then for every $\ell \in \mathcal{L}$ and prime $\lambda|\ell$ of $M_\pi$ as above, the Galois representation \( \rho_{\pi,\lambda} \otimes \ol{\bb{Q}}_l \) is strongly irreducible.
\end{theorem}

To conclude this section, we compare the extension $F_{1,\pi}/F$ to the monodromy group $G_{\rho_\lambda}$ itself.
Similarly to the above, by Theorem \ref{thm: formal char independent of l} there exists a finite extension \( F^\circ/F \) such that for every $\lambda$ there is an isomorphism
\[ \Gal(F^\circ/F) \cong G_{\rho_{\pi,\lambda}} / G_{\rho_{\pi,\lambda}}^\circ \]
induced via \( \rho_{\pi,\lambda} \). In particular, \( F^\circ \) is the minimal extension of \( F \) such that the restriction of \( \rho_{\pi,\lambda} \) to \( F^\circ \) has connected algebraic monodromy group.
The following lemma relates \( F^\circ \) to \( F_{1,\pi} \).

\begin{lemma} \label{lemma: F1,pi contains Fcirc}
There is an embedding \( F^\circ \hookrightarrow F_{1,\pi} \).
\end{lemma}
\begin{proof}
Since \( F\subset F_{1,\pi} \), and \( \Phi_\lambda|_F \) is isomorphic to a direct sum of three representations, one of which is equal to \( \rho_{\pi,\lambda} \), we have that if \( \Phi_\lambda|_{F_{1,\pi}} \) has connected algebraic monodromy group then the same is true for \( \rho_{\pi,\lambda}|_{F_{1,\pi}} \). Therefore, since \( F^\circ \) was the minimal such extension for \( \rho_{\pi,\lambda} \), we immediately have that \( F^\circ\subset F_{1,\pi} \).
\end{proof}

This lemma tells us that \( G_{\rho_{\pi,\lambda}|_{F_{1,\pi}}} = G_{\rho_{\pi,\lambda}}^\circ \), since by definition \( G_{\rho_{\pi,\lambda}}^\circ = G_{\rho_{\pi,\lambda}|_{F^\circ}} \) and \( F^\circ\subset F_{1,\pi} \) is of finite index.

\section{Tensor product factorisation} \label{section: tensor decomposition}

We fix some notation which will be in place throughout this section.
Let $(\pi,\chi)$ be as in the statement of Theorem \ref{theorem: xia's reduction}.
We recall from \Cref{section: Xia reduction} that we will view the weakly compatible system attached to $(\pi,\chi)$ as defined over a sufficiently large CM field $M_\pi$.
By assumption, $F$ is the maximal CM subfield of a certain finite extension $F_{1,\pi}/F^+$.
We fix a finite extension $N/M_\pi$ containing the image of every embedding $F_{1,\pi} \hookrightarrow N$ such that $N/\bb{Q}$ is Galois.
We recall the positive density set \(\mathcal{L} = \{\ell \mid [\Frob_\ell] = [c], \ell \text{ is unramified in } N \}\) of rational primes.

Our goal will be to reduce the proof of Theorem \ref{theorem: xia's reduction} to the case where the Lie algebra $\mf{g}_\lambda$ is simple for some prime $\lambda$.
The strategy will be to decompose a Galois representation attached to $\pi$ as a tensor product over $F$ itself and then apply potential automorphy to each of the tensorands.

We begin with a general result about decomposing Galois representations into tensor products.

\begin{proposition}\label{proposition: decompose Galois rep as tensor product}
    Let \( E \) be a number field. 
    Let \( \rho:G_E\to\GL_{n}(\Qlbar) \) be a strongly irreducible representation such that the following conditions hold.
    \begin{enumerate}
        \item $\rho$ is unramified outside of finitely many places.
        \item There exists an isomorphism of semisimple Lie algebras  $\Lie G_{\rho}^{\circ,\der}  \cong \prod_{i \in I} \mf{g}_i$, where $I$ is a finite index set and, for $i \in I$, $\mf{g}_i$ is a semisimple Lie algebra, under which the representation of $\Lie G_{\rho}^{\circ,\der}$ on $\Qlbar^n$ is isomorphic to $\boxtimes_{i \in I} V_i$ where, for $i \in I$, $V_i$ is a faithful irreducible representation of $\mf{g}_i$ of degree $n_i > 1$.
        \item $\rho(G_E) \subset \mathrm{im}(\prod_{i \in I} \GL_{n_i} \xrightarrow{\otimes} \GL_n)(\Qlbar)$, where the Kronecker map $\otimes$ is induced by the isomorphism $\boxtimes_{i \in I} V_i \cong \Qlbar^n$.
    \end{enumerate}
    Then there exist strongly irreducible representations \( \rho_i: G_{E}\to \GL_{n_i}(\Qlbar) \), unramified outside finitely many places, such that \(\rho \cong \otimes_{i \in I} \rho_i\) inducing isomorphisms $\Lie G_{\rho_i}^{\circ,\der} \cong \mf{g}_i$.
    The $\rho_i$ are uniquely defined, up to twisting by characters.

    Suppose in addition that the following conditions hold.
    \begin{enumerate}[resume]
        \item $\rho$ is de Rham (resp. crystalline) at each place $v|\ell$ of $E$.
        \item $\rho \cong r|_{G_E}$ for some polarizable Galois representation
        \[
        r: G_{E^{\mathrm{CM}}} \to \GL_n(\Qlbar),
        \]
        where $E^{\mathrm{CM}}$, the maximal CM subfield of $E$, is imaginary CM.
    \end{enumerate}
     Then we can additionally take each $\rho_i$ to be de Rham (resp. crystalline) at each place $v|\ell$.
\end{proposition}
\begin{proof}
The first part of the proposition follows from \cite[Lemma 3.2.13]{CEG22} and \cite[Proposition 5.3]{Conrad11} (with the former reference not making use of their assumption that the number field $E$ is imaginary CM).
The proof of uniqueness is the same as that of \cite[Lemma 3.2.14]{CEG22}.

Suppose now that $\rho$ is also de Rham and a restriction of a polarizable representation $r$ of $E^{\mathrm{CM}}$.
We may therefore apply \cite[Theorem 3.2.10]{patrikis2019variations} to the homomorphism $\prod_{i \in I} \GL_{n_i} \to \mathrm{im}(\prod_{i \in I} \GL_{n_i} \xrightarrow{\otimes} \GL_n)$, with \cite[Hypothesis 3.2.4]{patrikis2019variations} holding because of the polarizability assumption on $r$ and the conditions it imposes on the Hodge--Tate cocharacters of $\rho$.
If $\rho$ is furthermore crystalline, applying \cite[Proposition 6.5]{Conrad11} to the homomorphism $\prod_{i \in I} \GL_{n_i} \to \mathrm{im}(\prod_{i \in I} \GL_{n_i} \xrightarrow{\otimes} \GL_n)$, whose kernel is a torus, we can twist each $\rho_i|_{G_{F_v}}$ by some finite order character for each $v|\ell$ so that they are crystalline.
By global class field theory, we can therefore twist each $\rho_i$ by a finite order character to obtain a crystalline representation.
\end{proof}

Using Proposition \ref{proposition: decompose Galois rep as tensor product}, we will write $\rho_{\pi,\lambda}|_{G_{F_{1,\pi}}}$ as a tensor product of de Rham Galois representations for suitably chosen $\lambda$.
In order to obtain a similar decomposition of $\rho_{\pi,\lambda}$, we will show that the action of $\Gal(F_{1,\pi}/F)$ on these tensorands is trivial.
If we knew that these Galois representations were both: (1) defined over a number field, and (2) pure, then we could firstly show (using the Chebotarev density theorem) that the dual of each tensorand is a twist of the complex conjugated (on coefficients) tensorand (cf. \cite[Lemma 1.2(1)]{PatTay15}).
Using polarizability and regularity of $\rho_{\pi,\lambda}$ together with our choice of $\lambda$, we would be able to deduce that a generating set of $\Gal(F_{1,\pi}/F)$ indeed acts trivially on the tensorands (by an argument similar to the proof of \cite[Proposition 4]{Xia19}).

Unfortunately, we do not know (a priori) that these tensorands are pure.
Instead, we will use the existence of a sufficiently generic element in the image of $\rho_{\pi,\lambda}|_{G_{F_{1,\pi}}}$ together with purity of Hodge--Tate characters to establish that the dual of each tensorand is a twist of the complex conjugated tensorand.
This motivates what follows below.

\begin{lemma}\label{lemma: generic element of image}
Let $E$ be a number field.
Let $\rho: G_E \to \GL_n(\Qlbar)$ be a Galois representation unramified outside of a finite set $S$ of primes of $E$ which is Hodge--Tate at each $v|\ell$ with regular Hodge--Tate weights.
Then there exists a positive Dirichlet density set of primes $v \not\in S$ such that, letting $\alpha_1,\ldots,\alpha_n$ denote the eigenvalues of $\rho(\Frob_v)^{\mathrm{ss}}$, we have for every $i \neq j$
\[
\big(\frac{\alpha_i}{\alpha_j}\big)^n \neq 1.
\]
\end{lemma}
\begin{proof}
Let $\varphi: G_{\rho}^\circ \to \GL_n$ be the natural representation arising from $\rho$.
Consider the $n$-fold tensor product representation, $\varphi^{\otimes n}: G_{\rho}^\circ \to \GL_{n^n}$.
Suppose $h \in G_{\rho}^\circ(\Qlbar)$ is such that $h^{\otimes n} \in \varphi^{\otimes n}(G_{\rho}^\circ(\Qlbar))$ is $\Gamma$-regular in the sense of \cite[Definition 4.5]{LarPin92} (with respect to the faithful quotient of $G_{\rho}^\circ$ through which $\varphi^{\otimes n}$ factors).
As the formal character of $\varphi$ is multiplicity-free by Lemma \ref{lemma: formal char is multiplicity-free}, we see that $h$ is regular semisimple and that the $n$th powers of distinct eigenvalues of $h$ are also distinct.
Since $G_{\rho}^\circ \subset G_{\rho}$ is open, the result now follows from combining Chebotarev density with \cite[Proposition 7.2]{LarPin92} applied to $\varphi^{\otimes n}$ (noting that their proof goes through without needing to assume that $\varphi^{\otimes n}$ is valued in $\GL_{n^n}(\bb{Q}_\ell)$, nor that $\rho^{\otimes n}$ is part of a compatible system).
\end{proof}

Let $\Omega$ be an algebraically closed field.
Given (multi-)sets $A,B$ of elements of $\Omega^\times$, let
\[
AB = \{ab: a \in A, b \in B\}
\]
denote the multi-set of pairwise products of elements of $A$ and $B$.
In the case that $A = \{a\}$ is a singleton, we let $aB$ denote $AB$.
Define an equivalence relation $\sim$ on such multi-sets by requiring $A \sim B$ if and only if there exists $\xi \in \Omega^\times$ such that $\xi A = B$. 
\begin{lemma}
Let $C$ be a non-empty finite set of elements of $\Omega^\times$.
Then there exists only finitely many inequivalent decompositions of the form 
\[
C = A_1 \ldots A_k,
\]
with each $\# A_i > 1$.
Here two decompositions $C = A_1 \ldots A_k$ and $C = B_1 \ldots B_l$ are said to be equivalent if $k = l$ and, after rearranging, we have $A_i \sim B_i$ for every $i$.
\end{lemma}
\begin{proof}
By induction on $\# C$, it suffices to show that there are only finitely many inequivalent decompositions of the form $C = AB,$ with $\# A = a$ and $\# B = b$ fixed.
Let $\sigma: \{1,\ldots,a\} \times \{1,\ldots,b\} \to C$ be a bijection.
Consider the system of equations $x_i y_j = \sigma(i,j) \in \Omega^\times$ for $1 \leq i \leq a$ and $1 \leq j \leq b$.
If $x_i = \alpha_{i}$ and $y_j = \beta_{j}$ is a solution, then this solution is uniquely determined by $\alpha_1$, because $\alpha_i = \alpha_1 \frac{\sigma(i,1)}{\sigma(1,1)}$ and $\beta_j = \frac{\sigma(1,j)}{\alpha_1}$.
Hence if $x_i = \alpha_{i}'$ and $y_j = \beta_{j}'$ is a second solution, then $\{\alpha_i\} = \frac{\alpha_1}{\alpha_1'} \{\alpha_i'\}$ and similarly $\{\beta_j\} \sim \{\beta_j'\}$.
Since the data of a decomposition $C = AB$ together with a choice of bijections $A \cong \{1,\ldots,a\}$ and $B \cong \{1,\ldots,b\}$ uniquely determine a bijection $\sigma$, which in turn determines $A$ and $B$ uniquely up to equivalence, we see that the number of such inequivalent decompositions is at most $\frac{(ab)!}{a! b!}$.
\end{proof}

Consider the weakly compatible system $(\rho_{\pi,\lambda}|_{G_{F_{1,\pi}}})_\lambda$.
We can find a prime $v$ of $F_{1,\pi}$ such that the characteristic polynomial $Q_{v}(X) = \det(1- \rho_{\pi,\lambda}(\Frob_v)X)$ (independent of $\lambda$) has pairwise distinct roots $C$ with no ratio equal to an $n$th root of unity by applying Lemma \ref{lemma: generic element of image} to $\rho_{\pi,\lambda}|_{G_{F_{1,\pi}}}$ for any choice of $\lambda$.
Each $\alpha \in C$ lies in $\ol{\bb{Q}}$ and is a $q_v$-Weil number of some fixed weight $w$, where $q_v$ denotes the size of the corresponding residue field.

Let $C = A_1 \ldots A_k$ be a decomposition of $C$ with each $A_i$ a finite set of complex numbers of cardinality $n_i$ (on fixing an embedding $\ol{\bb{Q}} \hookrightarrow \bb{C}$).
Suppose that, for each $1 \leq i \leq k$, the product $\prod_{\alpha \in A_i} \alpha$ is algebraic and a $q_v$-Weil number of weight $w_i$.
If $\alpha \in A_i$, then $\alpha^{n_i}$ is also a $q_v$-Weil number of weight $w_i$,
since we can write \[\alpha^{n_i} = \prod_{\beta \in A_i} \beta \cdot \prod_{\beta \in A_i} \frac{\alpha \gamma}{\beta \gamma}\] for any choice of $\gamma \in A_1 \ldots \hat{A_i} \ldots A_k$, with $\alpha \gamma$ and each $\beta \gamma$ lying in $C$.

For each equivalence class of decomposition of $C$ admitting a representative with the above properties, choose such a representative $C = A_1 \ldots A_k$.
After replacing $M_\pi$ by a finite CM extension, we will suppose that $M_\pi$ contains every element of the form $\alpha^{n_i}$, where $i$ ranges from $1$ to $k$ and $\alpha \in A_i$.
Note that this may require redefining the extension $N/M_\pi$ and the set $\mathcal{L} = \{\ell: [\Frob_\ell] = [c], \ell \text{ is unramified in } N \}$ of \S \ref{section: Xia reduction} accordingly.

\begin{lemma} \label{lemma: existence of xi_i}
Suppose, in the setup above, that $C = A_1 \ldots A_k$ is a decomposition with each $A_i$ a finite set of complex numbers of cardinality $n_i$.
Suppose further that, for each $1 \leq i \leq k$, the product $\prod_{\alpha \in A_i} \alpha$ is algebraic and a $q_v$-Weil number of weight $w_i$.
Then for each $1 \leq i \leq k$ there exists $\xi_i \in \bb{Q}^\times$ such that for each $\alpha \in A_i$, the complex number $\xi_i \alpha^{n_i}$ lies inside $M_\pi \subset \bb{C}$ and moreover $\xi_i \alpha^{n_i}$ is a $q_v$-Weil number of some fixed weight $w_i \in \bb{Z}$.
\end{lemma}
\begin{proof}
By assumption, there exists an equivalent decomposition $C = A_1' \ldots A_k'$ (over $\bb{C}$) such that for each $1 \leq i \leq k$ and each $\alpha' \in A_i'$, the quantity $(\alpha')^{n_i}$ lies inside $M_\pi$.
For each $1 \leq i \leq k$, there exists $\delta_i \in \bb{C}^\times$ such that $\delta_i A_i = A_i'$.
Moreover, $\delta_i$ is algebraic because all elements of $A_i$ and $A_i'$ are algebraic.
It is now easily seen that $\xi_i = \delta_i^{n_i} \in \bb{Q}^\times$ will satisfy the conclusion of the lemma.
\end{proof}

\begin{proposition}\label{prop: tensor decomposition over CM field}
Let $\ell\in \mathcal{L}$ be sufficiently large such that $\ell \nmid q_v$ and $\rho = \rho_{\pi,\lambda}: G_F\to \GL_n(\Qlbar)$ is crystalline with Hodge--Tate weights in the Fontaine--Laffaille range for $\lambda| \ell$.
Suppose that we can write $\Lie G_{\rho}^{\circ,\der}  \cong \prod_{i \in I} \mf{g}_i$ and that under this identification, the representation $\Qlbar^{n}$ is isomorphic to $\boxtimes_{i \in I} V_i$ for some irreducible representation $V_i$ of a simple Lie algebra $\mf{g}_i$ of degree $n_i > 1$.
Then for every $\lambda | \ell$ and $i \in I$ there exists a strongly irreducible representation \( \rho_i: G_{F}\to \GL_{n_i}(\Qlbar)\) satisfying the following conditions.
\begin{enumerate}
    \item \(\rho \cong \otimes_{i \in I} \rho_i \) inducing isomorphisms $\Lie G_{\rho_i}^{\circ,\der} \cong \mf{g}_i$, under which the representation of $\mf{g}_i$ is isomorphic to $V_i$.
    \item Each $\rho_i$ is unramified outside of a finite set of primes.
    \item Each $\rho_i$ is crystalline at every place $w|\ell$ of $F$ with regular Hodge--Tate weights within the Fontaine--Laffaille range.
    \item Each $\rho_i$ is polarizable.
\end{enumerate}  
\end{proposition}
\begin{proof}
The proof takes inspiration from \cite[Section 7]{Xia19}.
By Lemma \ref{lemma: F1,pi contains Fcirc} and Proposition \ref{proposition: decompose Galois rep as tensor product}, we may write
\begin{equation} \label{eqn: tensor decomposition over F1,pi}
\rho|_{G_{F_{1,\pi}}} \cong \otimes_{i \in I} \Tilde{\rho}_i    
\end{equation}
with each $\Tilde{\rho}_i: G_{F_{1,\pi}} \to \GL_{n_i}(\Qlbar)$ assumed to be de Rham with regular Hodge--Tate weights.
Recall that we have fixed a place $v$ of $F_{1,\pi}$ for which $\rho|_{G_{F_{1,\pi}}}$ is unramified with $\rho|_{G_{F_{1,\pi}}}(\Frob_v)$ regular semisimple and whose set of eigenvalues $C$ does not contain any distinct elements with equal $n$th powers.
After twisting each $\Tilde{\rho}_i$ by a finite character if necessary, we may suppose that each $\Tilde{\rho}_i$ is also unramified at $v$.

From the definition of $\mathcal{L}$, the (unique) complex conjugation \( c \) on $M_\pi$ extends to a continuous automorphism of $M_{\pi,\lambda}$.
Let $\Tilde{c} \in \Gal(\Qlbar/\bb{Q}_\ell)$ be a lift of $c \in \Gal(M_{\pi,\lambda}/\mathbb{Q}_\ell) \hookrightarrow \Gal(M_{\pi}/\bb{Q})$.
Since $\rho$ is pure of some weight $w$, we know by the Chebotarev density theorem that
\[
{}^{\Tilde{c}} \rho \cong \rho^\vee \otimes \eps^{w}
\]
(cf. \cite[Lemma 1.2(1)]{PatTay15}).
Thus, uniqueness of the tensor product decomposition (\ref{eqn: tensor decomposition over F1,pi}) up to twists implies the existence of a permutation $\kappa$ of $I$ 
and characters $\varphi_i: G_{F_{1,\pi}} \to \Qlbar^\times$ (unramified at $v$) such that
\begin{equation} \label{eqn: kappa defining property}
 {}^{\Tilde{c}} \Tilde{\rho}_i \cong \Tilde{\rho}_{\kappa(i)}^\vee \otimes \varphi_i.   
\end{equation}

We claim that $\kappa = \id$.
Let $i \in I$ be arbitrary and let $A_i$ denote the set of eigenvalues of $\Tilde{\rho}_i(\Frob_v)$ over $\Qlbar$, so that $C = \prod_{i \in I} A_i$.
Since $\Tilde{\rho}_i$ is de Rham, the character $\det \Tilde{\rho}_i$ is Hodge--Tate and therefore pure by Lemma \ref{lemma: HT chars are pure}.
In particular, $\prod_{\alpha \in A_i} \alpha$ is a $q_v$-Weil number.
Consequently, for $\alpha \in A_i$, we have $\alpha^{n_i}$ is also a $q_v$-Weil number.
Lemma \ref{lemma: existence of xi_i} guarantees the existence of $\xi_i \in \ol{\bb{Q}}^\times$ such that, for every $\alpha \in A_i$, the quantity $\xi_i \alpha^{n_i}$ lies inside $M_\pi$ and is also a $q_v$-Weil number of some fixed weight $w_i \in \bb{Z}$.
Since $\Tilde{c}$ restricts to complex conjugation on $M_\pi$, we therefore have \[\Tilde{c}(\xi_i \alpha^{n_i}) = q_v^{w_i} \xi_i^{-1} \alpha^{-n_i}.\]
Let $\gamma_\alpha = \alpha \cdot \Tilde{c}(\alpha) \in \Qlbar^\times$.
Then \[\gamma_\alpha^{n_i} = \frac{q_v^{w_i}}{\xi_i \Tilde{c}(\xi_i)}\] is independent of $\alpha \in A_i$.
On the other hand, by (\ref{eqn: kappa defining property}) we have
\begin{equation} \label{eqn: equality of eigenvalues}
\{\gamma_\alpha \alpha^{-1}: \alpha \in A_i\} =  \{\varphi_i(\Frob_v) \beta^{-1} : \beta \in A_{\kappa(i)}\}.    
\end{equation}
Under the equality (\ref{eqn: equality of eigenvalues}), we see that each $\beta \in A_{\kappa(i)}$ is of the form 
\[
\beta = \alpha \cdot \frac{\varphi_i(\Frob_v)}{\gamma_{\alpha}}
\]
for some $\alpha \in A_i$.

Now suppose that $\kappa(i) \neq i$.
Let $\alpha_1 \neq \alpha_2$ be elements of $A_i$ and choose some $\delta_j \in A_j$ for every $j \in I \setminus \{i,\kappa(i)\}$.
The following are two distinct elements of $C$
\begin{align*}
\alpha_1 (\alpha_2 \cdot \frac{\varphi_i(\Frob_v)}{\gamma_{\alpha_2}})  \prod_{j \in I \setminus \{i,\kappa(i)\}} \delta_j   \\
\alpha_2 (\alpha_1 \cdot \frac{\varphi_i(\Frob_v)}{\gamma_{\alpha_1}})  \prod_{j \in I \setminus \{i,\kappa(i)\}} \delta_j
\end{align*}
whose $n$th powers are equal.
This contradicts our assumption on the eigenvalues of $\rho|_{G_{F_{1,\pi}}}(\Frob_v)$, establishing our claim.

Let $\sigma \in [c] = [\Frob_\ell] \subset \Gal(N / \bb{Q})$ be a complex conjugation.
Let $\ol{\sigma} = \sigma|_{F_{1,\pi}}$ and choose a prime $\eta|\ell$ of $F_{1,\pi}$ such that $\Frob_\eta =  \ol{\sigma}$.
Choose an embedding $\tau: F_{1,\pi} \hookrightarrow \Qlbar$ inducing the prime $\eta$.
Since $\ol{\sigma} = \Frob_\eta$ has order $2$, we see that $(F_{{1,\pi}})_\eta/\bb{Q}_\ell$ is the unique unramified quadratic subextension of $\Qlbar/\bb{Q}_\ell$, with the unique continuous extension of $\ol{\sigma}$ to $(F_{{1,\pi}})_\eta$ acting as the order two element of $\Gal((F_{{1,\pi}})_\eta/\bb{Q}_\ell)$.
Similarly, $M_{\pi,\lambda}/\bb{Q}_\ell$ is also a quadratic unramified extension on which $c$ acts non-trivially.
Consequently, we can identify $M_{\pi,\lambda}$ and $(F_{{1,\pi}})_\eta$ as subfields of $\Qlbar$ via $\tau$, and $\Tilde{c}|_{(F_{{1,\pi}})_\eta} = \Frob_\eta$.
We deduce that $\Tilde{c}^{-1} \circ \tau = \tau \circ \Frob_\eta^{-1} = \tau \circ \ol{\sigma}^{-1}$.

We therefore have equalities
\begin{align*}
    \mathrm{HT}_{\tau}(\Tilde{\rho}_i) 
    & = \mathrm{HT}_{\tilde{c}^{-1} \tau \ol{\sigma}^{-1}}(\Tilde{\rho}_i) \\
    &= \mathrm{HT}_{\tau}({}^{\Tilde{c}}\tilde{\rho}_i^{\ol{\sigma}}) \\
    &= \mathrm{HT}_{\tau}((\Tilde{\rho}_{i}^{\vee})^{\ol{\sigma}} \otimes \varphi_i^{\ol{\sigma}}),
\end{align*}
where the second equality follows from \cite[\S 1]{PatTay15}.
Since $\rho$ is polarizable, the representation $\rho|_{G_{F_{1,\pi}}}$ admits both $\Tilde{\rho}_i$ and $(\Tilde{\rho}_{i}^{\vee})^{\ol{\sigma}} \otimes \varphi_i^{\ol{\sigma}}$ as tensorands.
Regularity of $\rho|_{G_{F_{1,\pi}}}$ and uniqueness of the decomposition into tensorands implies that $ \Tilde{\rho}_i \cong (\Tilde{\rho}_{i}^{\vee})^{\ol{\sigma}} \otimes \chi_{\sigma,i}$ for some character
$\chi_{\sigma,i}: G_{F_{1,\pi}} \to \Qlbar^\times$.
We see that \begin{equation} \label{eqn: CSD tensorand}
\Tilde{\rho}_i^{\ol{\sigma}} \cong \Tilde{\rho}_{i}^{\vee} \otimes \chi_{\sigma,i}^{\ol{\sigma}}.
\end{equation}

Note that \( [c]\subset\Gal(N/\bb{Q}) \) is contained in the subgroup \( \Gal(N/F^+) \), and let $f: \Gal(N/F^+) \twoheadrightarrow \Gal(F_{1,\pi}/F^+)$ be the surjection given by restriction. We also note that \( \Gal(F_{1,\pi}/F) \) is precisely the subgroup of \( \Gal(F_{1,\pi}/F^+) \) generated by \( \{\sigma_1 \sigma_2: \sigma_i \in f([c]) \} \), as $F$ is exactly the maximal subfield of $F_{1,\pi}$ on which all complex conjugations coincide.
It follows from (\ref{eqn: CSD tensorand}) that for every pair of elements $\sigma_1, \sigma_2 \in f([c]) \subset \Gal(F_{1,\pi}/F^+)$ and lift $\gamma \in \Gal(\overline{\bb{Q}}/F)$ of $\sigma_1 \sigma_2$, we have isomorphisms
\[
\Tilde{\rho}_i^{\gamma} \cong \Tilde{\rho}_i \otimes \delta_{i,\gamma}
\]
for some characters $\delta_{i,\gamma}: G_{F_{1,\pi}} \to \Qlbar^\times$, whose product $\prod_{i \in I} \delta_{i,\gamma}$ is trivial by \cite[Lemma 4.8 (1)]{kret2022galois}.
Write $\Tilde{\rho_i}^\gamma = A_i (\Tilde{\rho}_i \otimes \delta_{\gamma,i}) A_i^{-1}$ for some matrix $A_i \in \GL_n(\Qlbar)$.
On the one hand, $\rho$ is strongly irreducible so by Schur's lemma there is a unique (up to scalar) matrix $A \in \GL_n(\Qlbar)$ satisfying
\[(\otimes_{i \in I} \Tilde{\rho}_i)^\gamma = A (\otimes_{i \in I} \Tilde{\rho}_i) A^{-1}.\]
Clearly, we can take $A = \otimes_{i \in I} A_i$, where the tensor product denotes the Kronecker product of matrices.
On the other hand, we have by definition that $\rho^\gamma = \rho(\gamma) \rho \rho(\gamma)^{-1}$.
By uniqueness, we deduce that under the isomorphism $\Qlbar^n \cong \otimes_{i \in I} \Qlbar^{n_i}$, the matrix $\rho(\gamma)$ is given by $\beta \otimes_{i \in I} A_i$, for some scalar $\beta$.
Since $\Gal(F_{1,\pi}/F) \twoheadrightarrow \Gal(F^\circ/F) \cong \pi_0(G_\rho)$ is generated by the images of such elements $\gamma$, we deduce that under the isomorphism $\Qlbar^n \cong \otimes_{i \in I} \Qlbar^{n_i}$, the group $\rho(G_F)$ lies in the image of the Kronecker map.

We may therefore apply Proposition \ref{proposition: decompose Galois rep as tensor product} to $\rho$ itself, from which we obtain a decomposition $\rho \cong \otimes_{i \in I} \rho_i$, with each $\rho_i: G_F \to \GL_{n_i}(\Qlbar)$ assumed to be crystalline of distinct Hodge--Tate weights.
For each $i$,
we have an isomorphism
\[
\rho_i^c \cong \rho_i^\vee \otimes \mu_i
\]
for some character $\mu_i: G_{F} \to \Qlbar^\times$ and choice of complex conjugation $c$.
This uniquely determines the character $\mu_i$ by \cite[Lemma 3.2.8]{CEG22}.
Viewing $\mu_i$ as the unique $1$-dimensional quotient of $\rho_i \otimes \rho_i^c$, we must have that $\mu_i^c = \mu_i$.
Thus $\mu_i$ extends to a character of $G_{F^+}$ and $\rho_i$ is polarizable.
Since $\rho$ has Hodge--Tate weights in the Fontaine--Laffaille range, the same must also be true of each $\rho_i$.
\end{proof}

\begin{corollary} \label{cor: tensorands are potentially automorphic}
There exists a set $\mathcal{L}'$ of rational primes of positive Dirichlet density with the following properties.
Suppose that $\ell \in \mathcal{L}'$ and $\lambda|\ell$ is a place of $M_\pi$.
The Galois representation $\rho = \rho_{\pi,\lambda}$ is strongly irreducible and decomposes accordingly as a tensor product
\[
\rho \cong \otimes_{i \in I} \rho_i
\]
where $\rho_i: G_F \to \GL_{n_i}(\Qlbar)$ satisfy the following conditions.
\begin{enumerate}
    \item There is a natural identification 
    $\Lie G_{\rho_i}^{\circ,\der} \cong \mf{g}_i$ with the action on $\Qlbar^{n_i}$ isomorphic to $V_i$.
    \item Each $\rho_i$ is unramified outside of a finite set of primes
    \item Each $\rho_i$ is crystalline at each $v|\ell$ with regular Hodge--Tate weights lying in the Fontaine--Laffaille range.
    \item Each $\rho_i$ is polarizable.
    \item Each representation $\overline{\rho_i}|_{G_{F(\zeta_\ell)}}$ is absolutely irreducible.
    \item For any finite Galois extension $F^{\mathrm{(avoid)}}/F$, there exists a finite CM extension $K/F$ linearly disjoint from $F^{\mathrm{(avoid)}}/F$ such that $\rho_i|_{G_K}$ is automorphic: there exists a polarized regular algebraic cuspidal automorphic representation $\pi_i$ of $\GL_{n_i}(\bb{A}_K)$ with an associated $\ell$-adic Galois representation $\rho_{\pi_i,\ell}$ isomorphic to $\rho_i|_{G_K}$.
\end{enumerate}
Consequently, on replacing $M_\pi$ by a finite extension if necessary, for every finite place $\Tilde{\lambda}$ of $M_{\pi}$, there exists Galois representations $r_{\pi_i, \Tilde{\lambda}}: G_K \to \GL_n(\overline{M_{\pi,\Tilde{\lambda}}})$ attached to $\pi_i$ such that $\rho_{\pi,\Tilde{\lambda}}|_{G_K} \cong \otimes_{i \in I} r_{\pi_i, \Tilde{\lambda}}$.
\end{corollary}

\begin{proof}
The first four parts follow from Proposition \ref{prop: tensor decomposition over CM field} applied to $\ell \in \mathcal{L}$ sufficiently large.
By Proposition \ref{proposition: density 1 residual irreducibility of summands}, we can choose a further subset $\mathcal{L}' \subset \mathcal{L}$ of positive Dirichlet density such that for each $\ell \in \mathcal{L}'$, we have $\ell \geq \max(2n+2,11)$, $\ell$ is unramified in $F$ and $\overline{\rho}|_{G_{F(\zeta_\ell)}}$ is absolutely irreducible.
Since $\overline{\rho}|_{G_{F(\zeta_\ell)}} \cong \otimes_{i \in I} \overline{\rho_i}|_{G_{F(\zeta_\ell)}}$, the fifth part now follows.

Let $J = \{i \in I: \rho_i \text{ is totally odd}\}$.
If $i \in I$ with $\dim \rho_i$ odd then $i \in J$ by Lemma \ref{lemma: oddness}.
By applying \cite[Theorem 1.2]{Hui23} to the weakly compatible system $(\rho_{\pi,\lambda} \otimes \rho_{\pi,\lambda}^\vee)_\lambda$, we may suppose, after excluding finitely many primes from $\mathcal{L}'$, that $\Sym^2 \overline{\rho_i}$ is irreducible whenever $i \in I$ with $\dim \rho_i = 2$.
This is because $\mathrm{ad}^0 \rho_i$ will occur as an irreducible summand of $\rho_{\pi,\lambda} \otimes \rho_{\pi,\lambda}^\vee$ with type $A_1$ image and loc. cit. guarantees that $\mathrm{ad}^0 \ol{\rho_i}$, a twist of $\Sym^2 \overline{\rho_i}$, will be irreducible.
Since $\ol{\rho_i}|_{G_{F(\zeta_\ell)}}$ is also irreducible, if $\Sym^2 \overline{\rho_i}|_{G_{F(\zeta_\ell)}}$ were reducible then $\overline{\rho_i}$ would have dihedral image, just as in the proof of \cite[Theorem 3.2]{CEG22}, contradicting the irreducibility of $\Sym^2 \overline{\rho_i}$, which therefore forces $\Sym^2 \overline{\rho_i}|_{G_{F(\zeta_\ell)}}$ to be irreducible.
Moreover, on excluding finitely many more primes (dependent only on $\pi$) from $\mc{L}'$ again, we have that $\Sym^2 \rho_i$ is Fontaine--Laffaille.
We therefore see by Lemma \ref{lemma: oddness} that $i \in J$.

It follows from our assumption that if $4|n$ then $n = 4p$ for some prime number $p$, that $|I \setminus J| \leq 1$.
Since $\rho \cong \otimes_{i \in I} \rho_i$ itself is totally odd, it follows from Lemma \ref{lemma: oddness} that $I = J$.
We see that the collection $(\rho_i)_{i \in I}$ satisfies the hypotheses of Theorem \ref{thm: blggt potential automorphy}.
Thus, there is a finite CM extension $K/F$ such that each $\rho_i|_{G_K}$ is automorphic as in the statement of the corollary and therefore each $\rho_i|_{G_K}$ gives rise to a weakly compatible system.
\end{proof}

\begin{remark}
In the setup of Corollary \ref{cor: tensorands are potentially automorphic}, by the Khare--Wintenberger method each $\rho_i$ is actually part of a weakly compatible system over $F$ itself (see \cite[Theorem 5.5.1]{BLGGT14}).
By Chebotarev density, it follows that for \emph{every} place $\lambda$, the Galois representation $\rho_{\pi,\lambda}$ can be written as a tensor product of $n_i$-dimensional Galois representations. 
\end{remark}

Using Corollary \ref{cor: tensorands are potentially automorphic}, we can reduce the proof of Theorem \ref{theorem: xia's reduction} to showing the following theorem.

\begin{theorem} \label{thm: reduction to simple tensorands}
    Let \( F \) be an imaginary CM field. Let \( (\pi,\chi) \) be a regular algebraic polarized cuspidal automorphic representation of \( \GL_n(\bb{A}_F) \) with associated weakly compatible system of Galois representations \( (\rho_{\pi,\lambda})_\lambda \) of \( G_F \) defined over an imaginary CM field \( M_\pi \).
    Suppose that \( 7\nmid n \) and, if \( 4\mid n \), then $n = 4p$ for some prime number $p$.
    Suppose that there exists a prime $\lambda_0$ of $M_\pi$ for which
    \begin{enumerate}
        \item $\rho_{\pi,\lambda_0}$ is strongly irreducible,
        \item $0 \neq \Lie G_{\rho_{\pi,\lambda_0}}^{\circ, \der}$ is simple,
        \item $\lambda_0 | \ell_0$ prime with $\ell_0 \geq 2n+2$ unramified in $F$,
        \item $\rho_{\pi,\lambda_0}$ is crystalline with $\HT_\tau(\rho_{\pi,\lambda_0}) \subset [a_\tau, a_\tau + 2 \ell_0 - 4]$ for each $\tau:~F~\hookrightarrow~\Qlzerobar$,
        \item $\ol{\rho_{\pi,\lambda_0}}|_{G_{F(\zeta_{\ell_0})}}$ is irreducible.
    \end{enumerate}
    Then there exists a Dirichlet density 1 set $\mathcal{L}''$ of rational primes such that for all places \( \lambda \) dividing a prime in $\mc{L}''$, \( \rho_{\pi,\lambda} \) is strongly irreducible.    
\end{theorem}

\begin{proof}[Proof of Theorem \ref{theorem: xia's reduction}, assuming Theorem \ref{thm: reduction to simple tensorands}]
Let $\pi$ be as in the statement of Theorem \ref{theorem: xia's reduction} and let $\mc{L}'$ be as in Corollary \ref{cor: tensorands are potentially automorphic} applied to $\pi$.
There exists $k \geq 1$ and an infinite subset $\mc{L}'' \subset \mc{L}'$ such that for every $\ell \in \mc{L}''$ and some (in fact, any) $\lambda|\ell$, the number of tensorands in the decomposition of $\rho_{\pi,\lambda}$ obtained via Proposition \ref{prop: tensor decomposition over CM field} is equal to $k$.
Fix $\ell' \in \mc{L}''$ and let $\lambda' | \ell'$.
Corollary \ref{cor: tensorands are potentially automorphic} applied to \( \lambda' \) produces a CM extension $K/F$ such that for every prime $\lambda$ of $M_\pi$ (extending $M_\pi$ if necessary and replacing $\lambda$ with any prime lying above), there are isomorphisms
\[
\rho_{\pi,\lambda}|_{G_K} \cong \otimes_{i \in I} \rho_{\pi_i,\lambda}
\]
where each $\pi_i$ is a polarized regular algebraic cuspidal automorphic representation of $\GL_{n_i}(\bb{A}_K)$.

For $i \in I$ and $\lambda$ a prime of $M_\pi$, let $\mf{g}_{i,\lambda} = \mf{g}_{\rho_{\pi_i,\lambda}}$.
We claim that $\mf{g}_{\rho_{\pi,\lambda}} \cong \prod_{i \in I} \mf{g}_{i,\lambda}$.
We know that $\mf{g}_{\rho_{\pi,\lambda}}$ is isomorphic to a subalgebra of $\prod_{i \in I} \mf{g}_{i,\lambda}$ surjecting onto each $\mf{g}_{i,\lambda}$.
On the other hand, we know that both $\rk \mf{g}_{\rho_{\pi,\lambda}}$ and $\rk \mf{g}_{i,\lambda}$ is independent of $\lambda$ by Theorem \ref{thm: formal char independent of l}.
Since the place $\lambda'$ satisfies
\[
\rk \mf{g}_{\rho_{\pi,\lambda'}} = \sum_{i \in I} \rk \mf{g}_{i,\lambda'},
\]
we must also have
\[
\rk \mf{g}_{\rho_{\pi,\lambda}} = \sum_{i \in I} \rk \mf{g}_{i,\lambda},
\]
from which the claim follows by Lemma \ref{lemma: Goursat consequence} stated below.

We next claim that we can find $\ell_0 \in \mc{L}''$ and $\lambda_0|\ell_0$ a prime of $M_\pi$ such that, for $i \in I$, $\pi_i$ satisfies the hypotheses of Theorem \ref{thm: reduction to simple tensorands} applied to $\lambda_0$.
(We cannot necessarily take $\ell_0 = \ell'$, since we do not know whether $\ell'$ is unramified in $K$.)
From what we already know about $\pi$, if we take $\ell_0 \in \mc{L}''$ to be sufficiently large, then we only need to check that each $\mf{g}_{i,\lambda_0}$ is simple.
By definition of $\mc{L}'$, we can write $\rho_{\pi,\lambda_0} \cong \otimes_{j \in J} \rho_j$ with $\#J = k$, each $\mf{g}_{\rho_j}$ simple and $\mf{g}_{\rho_{\pi,\lambda_0}} \cong \oplus_{j \in J} \mf{g}_{\rho_j}$.
There are isomorphisms $\otimes_{i \in I} \rho_{\pi_i,\lambda_0} \cong \rho_{\pi,\lambda_0}|_{G_K} \cong \otimes_{j \in J} \rho_j|_{G_K}$.
We saw above that $\mf{g}_{(\rho_{\pi,\lambda_0}|_{G_K})} \cong \prod_{i \in I} \mf{g}_{i,\lambda_0}$.
Moreover by Theorem \ref{thm: formal char independent of l}, each $\mf{g}_{i,\lambda_0}$ is non-zero since $\mf{g}_{i,\lambda'} \neq 0$.
Since $\#I = \#J = k$ and
\[
\prod_{j \in J} \mf{g}_{\rho_j} \cong
\mf{g}_{\rho_{\pi,\lambda_0}} \cong \mf{g}_{(\rho_{\pi,\lambda_0}|_{G_K})} \cong \prod_{i \in I} \mf{g}_{i,\lambda_0},
\]
each $\mf{g}_{i,\lambda_0}$ must be simple, as claimed.

Applying Theorem \ref{thm: reduction to simple tensorands} to $\pi_i$ and $\lambda_0$, we can find a Dirichlet density \( 1 \) set of primes \( \mc{L}_i \) such that \( \rho_{\pi_i,\lambda} \) is strongly irreducible for $\lambda| \ell \in \mathcal{L}_i$.
Let \( \mc{L} := \cap_{i \in I} \mc{L}_i \), also of Dirichlet density \( 1 \), and let $\lambda| \ell \in \mathcal{L}$.
Each $\rho_{\pi_i,\lambda}$ is strongly irreducible, so that each $\mf{g}_{i,\lambda}$ acts irreducibly on $\Qlbar^{n_i}$.
We saw above that $\mf{g}_{\rho_{\pi,\lambda}} \cong \prod_{i \in I} \mf{g}_{i,\lambda}$.
We deduce that $\mf{g}_{\rho_{\pi,\lambda}}$ acts irreducibly on $\Qlbar^n$, since the representation is isomorphic to an external tensor product of irreducible representations of each $\mf{g}_{i,\lambda}$.
\end{proof}

\begin{lemma} \label{lemma: Goursat consequence}
Let $\mf{h} = \prod_{i=1}^k \mf{h}_i$ be a semisimple Lie algebra and let $\mf{f} \subset \mf{h}$ be a semisimple subalgebra such that the induced map $\mf{f} \to \mf{h}_i$ is surjective for every $1 \leq i \leq k$.
If there is an equality of ranks $\rk \mf{f} = \rk \mf{h}$, then there is an equality of Lie algebras $\mf{f} = \mf{h}$.
\end{lemma}
\begin{proof}
We proceed by induction on $k$, first proving the lemma in the case where $k = 2$ (since the case $k=1$ is obvious).
A version of Goursat's lemma for Lie algebras states that there exists ideals $I_i \leq \mf{h}_i$ and an isomorphism $\varphi: \mf{h}_1/I_1 \xrightarrow{\sim} \mf{h}_2/I_2$ such that
\[\mf{f} = \{(x,y) \in \mf{h}_1 \times \mf{h}_2: \varphi(x + I_1) = y + I_2\}.\]
Since each $\mf{h}_i$ is semisimple, each ideal $I_i$ is a semisimple subalgebra and direct product factor of $\mf{h}_i$.
There is a short exact sequence
\[
0 \to I_1 \times I_2 \to \mf{f} \to \mf{h}_1/I_1 \to 0,
\]
from which we see that $ \rk \mf{h}_1 +  \rk I_2 =  \rk \mf{f} = \rk \mf{h}_1 + \rk \mf{h}_2 $.
We must therefore have $\mf{h}_2 = I_2$ and the result in this case now easily follows.

Now suppose that $k>2$.
Let $\mf{h}' = \prod_{i=1}^{k-1} \mf{h}_i$ and let $\pi: \mf{h} \to \mf{h}'$ denote the natural projection.
Let $\mf{f}' = \pi(\mf{f})$.
Then the map $\mf{f}' \to \mf{h}_i$ is surjective for every $1 \leq i \leq k-1$ and we have \[ \rk \mf{h}' = \rk \mf{f} - \rk \mf{h}_k \leq \rk \mf{f'} \leq \rk \mf{h'},\]
so equalities must hold throughout.
By induction, we can therefore apply the lemma to $\mf{f}' \subset \mf{h}'$ to deduce that $\mf{f}' = \mf{h}'$.
We see that $\mf{f}$ surjects onto both $\mf{h}'$ and $\mf{h}_k$, so applying the lemma once again, we see that $\mf{f} = \mf{h}$, which completes the proof.
\end{proof}

\section{Irreducibility for an elementary class of compatible systems} \label{section: simple irreducibility}

This section is dedicated to the proof of Theorem \ref{thm: reduction to simple tensorands}.

\subsection{Equalities of formal characters in type A} \label{subsect: type A formal characters}

We begin by recalling some notions regarding formal characters of Lie algebra representations, following \cite[\S 2]{hui2013monodromy}.
Suppose that we have, for $i=1,2$, complex semisimple Lie algebras $\mf{g}_i$ and $n$-dimensional faithful representations $V_i$ of $\mf{g}_i$. 
We can choose Cartan subalgebras $\mf{t}_i \subset \mf{g}_i$ and decompose $V_i|_{\mf{t}_i}$ to obtain the character of $V_i$.
We call the corresponding element of $\bb{Z}[\Lambda_i]$, where $\Lambda_i$ is the weight lattice of $\mf{g}_i$, the \emph{formal character} of the representation $\mf{g}_i \to \gllie(V_i)$.
The representations $V_i$ are said to have the same formal character if there is an isomorphism $F: \mf{t}_1^* \xrightarrow{\sim} \mf{t}_2^*$ sending the character of $V_1$ to the character of $V_2$. If \( \rho:G_F\to\GL_n(\Qlbar) \) is a continuous representation, and \( T'\subset T\subset \GL_n \) is the formal bi-character of \( G_\rho\subset \GL_n \), then \( \Lie(T')\subset \Lie(\GL_n) \) is the formal character of the induced representation of $\mf{g}_{\rho}$ in this setting.

If $\sum_j \alpha_j \in \bb{Z}[\Lambda_1]$ is the formal character of $\mf{g}_1$, we can define an inner product on $(\Lambda_1 \otimes \bb{R})^*$ by setting
\[
(x_1,x_2)_1 = \sum_j \alpha_j(x_1) \alpha_j(x_2).
\]
Let $\langle-,-\rangle_1$ denote the dual inner product on $(\Lambda_1 \otimes \bb{R})$.
Similarly, we obtain an inner product $\langle-,-\rangle_2$ on $\Lambda_2 \otimes \bb{R}$.
Since $\langle-,-\rangle_i$ is invariant with respect to the action of the Weyl group $W_{\mf{g}_i}$ of $\mf{g}_i$, this inner product respects the decomposition of the root system of $\mf{g}$ into irreducible root systems and restricts to the unique (up to scalar) Weyl group invariant inner product on the weight space of each simple factor of $\mf{g}_i$.

Suppose from now on that the formal characters of $V_1$ and $V_2$ are the same.
The isomorphism $F$ induces an isometry of the above inner product spaces.
Moreover, if $\alpha, \beta \in \Lambda_{1}$ are weights appearing in the character of $V_1$ then 
\begin{equation} \label{eqn: equality of inner products of weights}
\langle \alpha,\beta \rangle_1 = \langle F(\alpha),F(\beta) \rangle_2.    
\end{equation}
In order to constrain the possibilities for $\mf{g}_2$ in terms of $\mf{g}_1$, Hui reduces to the case where each $\mf{g}_i$ is of type $A$, using the following lemma.

\begin{lemma}[{\cite[Lemma 2.5]{hui2013monodromy}}] \label{lemma: maximal type A subalgebras}
If $\mf{g}$ is a semisimple Lie algebra then $\mf{g}$ contains a subalgebra $\mf{h} \subset \mf{g}$ of the same rank such that every simple factor of $\mf{h}$ is of type $A$.
Moreover, for $n=6$ or $n \geq 9$, the number of simple factors of type $A_n$ of $\mf{h}$ and $\mf{g}$ coincide.
\end{lemma}

We note that the statement of \cite[Lemma 2.5]{hui2013monodromy} contains a small mistake, in that \textit{loc. cit.} states that the number of $A_4$ factors of $\mf{g}$ and $\mf{h} \subset \mf{g}$ must coincide, but this doesn't hold for $\sllie_5^2 \subset \mf{e}_8$, for example (with only the parity of the number of $A_4$ factors preserved in general).

Lemma \ref{lemma: maximal type A subalgebras} allows one to assume that each $\mf{g}_i$ is of type $A$, since restricting a representation to a subalgebra of the same rank yields an isomorphism of formal characters.
Via an intricate analysis of the possibilities for the inner products and norms of roots of $\mf{g}_i$, Hui proves the following theorem. 

\begin{theorem}[{\cite[Theorems 2.14, 2.17]{hui2013monodromy}}] \label{thm: formal char type a subalgebra equivalence} 
Let $\mf{g}$ (resp. $\mf{g}'$) be a semisimple complex Lie algebra of type $A$ and let $V$ (resp. $V'$) be a faithful representation of $\mf{g}$ (resp. $\mf{g}'$).
Suppose that $V$ and $V'$ have the same formal characters.
Let $a_n$ (resp. $a_n'$) denote the number of type $A_n$ factors of $\mf{g}$ (resp. $\mf{g}'$).
Then $a_4 \equiv a_4' \mod 2$ and $a_n = a_n'$ for $n = 6$ or $n \geq 9$.
\end{theorem}

The following result constitutes part of the proof of Theorem \ref{thm: formal char type a subalgebra equivalence}. 

\begin{lemma}[{\cite[\S 2.3]{hui2013monodromy}}] \label{lemma: type A_n isomorphic for most n}
Suppose that we are in the setup of Theorem \ref{thm: formal char type a subalgebra equivalence} with $\mf{g}$ simple of type $A_n$.
If $n \in \{4,5,6\}$ or $n \geq 9$, then $\mf{g} \cong \mf{g}'$.
\end{lemma}

Suppose that $\mf{g} = \sllie_{n+1}$.
Let $e_1,\ldots,e_{n+1}$ denote the weights of the standard representation of $\mf{g}$, which span $\Lambda_{\mf{g}} \otimes \bb{R}$ and are subject to the relation $\sum_{i} e_i = 0$. 
Let $\langle-,-\rangle$ be a $W_{\mf{g}}$-invariant inner product on $\Lambda_{\mf{g}} \otimes \bb{R}$, normalised so that $\langle e_i, e_i \rangle = \frac{n}{n+1}$.
We then have formulas (\cite[\S 2.7]{hui2013monodromy})
\begin{align} 
\langle e_i, e_j \rangle &= \delta_{i,j} - \frac{1}{n+1} \label{eqn: A_n inner products of weights} \\
||\sum_i a_i e_i ||^2 &= \frac{\sum_i n a_i^2 -  \sum_{i \neq j} a_i a_j}{n+1}.  \label{eqn: A_n norms of weights}  
\end{align}
In the following two lemmas, let $[m] = \{1,\ldots,m\}$.

\begin{lemma} \label{lemma: alternating power inner products}
Let $V$ be the irreducible representation $\Lambda^a \Std$ of $\sllie_{n+1}$, for some $1 \leq a \leq n$.
\begin{enumerate}
\item The weights of $V$ (with multiplicity) are given by
\[ \mathcal{S}_V = \{\sum_{i \in I} e_i | I \subset [n+1], \# I = a \}.\]
\item Each $v \in \mathcal{S}_V$ satisfies $||v||^2 = \frac{a(n+1-a)}{n+1}$.
\item If $v \in \mathcal{S}_V$ then
\begin{align*}
\max_{u \in \mathcal{S}_V \setminus \{v\}} \langle u, v \rangle &= ||v||^2 - 1,    \\
\min_{u \in \mathcal{S}_V \setminus \{v\}} \langle u, v \rangle &= ||v||^2 - \min(a,n+1-a).
\end{align*}
\end{enumerate}
\end{lemma}
\begin{proof}
The first part is standard.
Let $I,J \subset [n+1]$ and set $u = \sum_{i \in I} e_i$ and $v = \sum_{j \in J} e_j$.
Let $k = \#( I \cap J)$.
Then $\langle u, v \rangle = \frac{n k - (a^2 - k)}{n+1} = \frac{(n+1)k - a^2}{n+1}$ by (\ref{eqn: A_n norms of weights}).
The second part follows and we see that $\langle u , v \rangle = ||v||^2 - (a-k)$.
Fixing $v$ and varying $u$, we see that $k$ may take any integer value in $[\max(0,2a-n-1),a]$, from which the final part now follows.
\end{proof}

The next lemma is straightforward.

\begin{lemma} \label{lemma: symmetric power weights}
Let $V$ be the irreducible representation $\Sym^a \Std$ of $\sllie_{n+1}$, for some $a \geq 1$.
\begin{enumerate}
\item The weights of $V$ (with multiplicity) are given by
\[\mathcal{S}_V = \{\sum_{i \in [a]} e_{f(i)} | f: [a] \to  [n+1] \}.\]
\item  The set $\mathcal{S}_V$ has $n+1$ elements of maximal norm, given by $a e_i$ for $1 \leq i \leq n+1$.
\end{enumerate}
\end{lemma}

The following lemma will allow us to prove a variant of Lemma \ref{lemma: type A_n isomorphic for most n} for every $n$ and a large class of $V$.

\begin{lemma} \label{lemma: bound weights of maximal norm}
Let $\mf{g} = \prod_{i=1}^k \mf{q}_i$ be a semisimple Lie algebra, with each $\mf{q}_i$ simple of type $A_{n_i}$.
Let $r$ denote the rank of $\mf{g}$ and let $V$ be a faithful representation of $\mf{g}$.
Let $\langle-,-\rangle$ be any $W_{\mf{g}}$-invariant inner product on $\Lambda_{\mf{g}} \otimes \bb{R}$ (giving rise to a norm \( \norm{\, \cdot\, } \)).
Let $W_{\max}$ denote the set of weights of $V$ of maximal norm.
If $W_{\max}$ spans $\Lambda_{\mf{g}} \otimes \bb{R}$, then $\# W_{\max} \geq r+1$, with equality holding possibly only if $k=1$.
\end{lemma}
\begin{proof}
By our assumptions on $W_{\mathrm{max}}$ we can choose a partition $\mathcal{P}$ of $\{1,\ldots,k\}$ such that for each $I \in \mathcal{P}$ there exists an irreducible summand $U_I \leq V$ such that $U_I \cong \boxtimes_{i=1}^k U_{I,i}$ with $U_{I,i}$ non-trivial for $i \in I$ and $U_I$ admitting a weight in $W_{\max}$.
If $i \in I$, then $U_{I,i}$ admits at least $n_i+1$ weights of norm maximal with respect to the restriction of $||\cdot ||$ to $\Lambda_{\mf{q}_i} \otimes \bb{R}$ (since the size of the orbit of any non-zero weight of $\mf{q}_i$ under the action of $W_{\mf{q}_i} \cong S_{n_i+1}$ is given by a non-trivial multinomial coefficient, and is hence at least $n_i+1$).
It follows that \[\# W_{\max} \geq \sum_{I \in \mathcal{P}} \prod_{i \in I} (1 + n_i) \geq \sum_{I \in \mathcal{P}} (1 + \sum_{i \in I} n_i) \geq r + 1,\] 
with equality holding throughout only if $k=1$.
\end{proof}

\begin{proposition} \label{prop: symmetric power type A are iso}
Suppose that we are in the setup of Theorem \ref{thm: formal char type a subalgebra equivalence}.
Suppose additionally that $\mf{g} \cong \sllie_{n+1}$ with $V$ isomorphic to $\Sym^a \Std$ (or its dual), for some $a \geq 1$.
Then $\mf{g} \cong \mf{g}'$.
\end{proposition}
\begin{proof}
We only do the case of $V \cong \Sym^a \Std$, with the dual case being similar.
The representation $V$ admits exactly $n+1$ weights of maximal norm with respect to $||\cdot ||$ by Lemma \ref{lemma: symmetric power weights}.
We see from (\ref{eqn: equality of inner products of weights}) that $V'$ must also have exactly $n+1$ weights of maximal norm (with respect to the inner product induced from $V'$) and that these weights must span $\Lambda_{\mf{g}'} \otimes \bb{R}$.
By Lemma \ref{lemma: bound weights of maximal norm}, we see that $k=1$ and the result follows.
\end{proof}

\subsection{Formal characters and complex conjugation}\label{subsect: formal chars and complex conj}

Suppose, throughout the rest of this section, that we are in the setup of Theorem \ref{thm: reduction to simple tensorands}.
Let $\rho = \rho_{\pi,\lambda_0}$ satisfying $\rho^c \cong \rho^\vee \otimes \rho_\chi|_{G_F}$.
Let $\mf{g}_\lambda = \Lie G_{\rho_\lambda}^{\circ, \der}$ and write $\mf{g} = \mf{g}_{\lambda_0}$, a simple Lie algebra.
Let $V_\lambda$ denote the representation of $\mf{g}_\lambda$ on $\Qlbar^n$, and write $V$ for $V_{\lambda_0}$, a multiplicity-free representation of $\mf{g}$ by Lemma \ref{lemma: formal char is multiplicity-free}.
Let $\mathcal{L}_\pi$ be a positive density set of rational primes as in Theorem \ref{thm: irreducibility from polarized summands}.

We recall a result of \cite{Xia19} about the formal characters of weakly compatible systems and the action of complex conjugation, applicable to our setup.
As in \cite[\S 8]{Xia19}, choose a certain abelian, semisimple, weakly compatible system $(\eps_{\mathfrak{m},\lambda})_\lambda$ of $F_{1,\pi}$ defined over $M_\pi$.
Define an auxiliary weakly compatible system
\[
(\beta_\lambda := (\Ind_{F^+}^F \rho_{\pi,\lambda}) \oplus  \rho_{\chi,\lambda} \oplus (\Ind_{F^+}^{F_{1,\pi}} \eps_{\mathfrak{m},\lambda}))_\lambda.
\]
Then $H_\lambda := G_{\beta_\lambda}^{\circ,\der}$, the identity component  of the derived group of the Zariski closure of the image of $\beta_\lambda$, naturally coincides with $G^{\circ,\der}_\lambda$.
Let $c_\lambda$ denote the image of a complex conjugation $c \in \Gal(\ol{\bb{Q}}/F^+)$ under $\beta_\lambda$. Choose a maximal torus and Borel subgroup $T_\lambda \subset B_\lambda \subset H_\lambda$ stabilized by $c_\lambda$, as in \cite[Proposition 8]{Xia19}. 
Let $W_\lambda \subset X^*(T_\lambda)$ denote the size $n$ (multi-)set of weights of $T_\lambda$ induced by $\rho_\lambda$.
Let $W_{c_\lambda} = \{w + w^{c_\lambda} : w \in W_\lambda\}$, a multi-subset of $X^*(T_\lambda)^{c_\lambda = 1}$.

\begin{proposition}[{\cite[Proposition 17]{Xia19}}] \label{prop: complex conjugation formal character action}
For each pair of primes $\lambda, \lambda'$ of $M_\pi$, there exists an isomorphism $i_{\lambda,\lambda'}: X^*(T_\lambda) \xrightarrow{\sim} X^*(T_{\lambda'})$ such that $i_{\lambda,\lambda'}(W_\lambda) = W_{\lambda'}$, and an isomorphism $r_{\lambda,\lambda'}: X^*(T_\lambda)^{c_\lambda = 1} \xrightarrow{\sim} X^*(T_{\lambda'})^{c_{\lambda'} = 1}$ such that $r_{\lambda,\lambda'}(W_{c_\lambda}) = W_{c_{\lambda'}}$.
\end{proposition}

We can use Proposition \ref{prop: complex conjugation formal character action} to obtain criteria for polarizability of summands of Galois representations.

\begin{lemma} \label{lemma: conjugate self dual summand criterion}
Let $\lambda$ be a prime of $M_\pi$ and let $U$ be an irreducible summand of $\rho_{\pi,\lambda}$.
Suppose that there exists a weight $w$ of $U$, viewed as a representation of $\mf{g}_\lambda$, such that $-w$ is a weight of $U^c$.
Then $U$ is polarizable.
\end{lemma}
\begin{proof}
The formal bi-character of $\rho_{\pi,\lambda}$ is multiplicity-free by Lemma \ref{lemma: formal char is multiplicity-free}.
Let $X$ (resp. $Y$) be the unique irreducible $G_F$-subrepresentation of $\rho_{\pi,\lambda}^c$ (resp. $\rho_{\pi,\lambda}^\vee$) which, when viewed as a representation of $\mf{g}_\lambda$, admits $-w$ as a weight.
By uniqueness, we must have $X \cong U^c$ and $Y \cong U^\vee$.
It follows that, under the isomorphism $\rho_{\pi,\lambda}^c \cong \rho_{\pi,\lambda}^\vee \otimes \rho_{\chi,\lambda}$, $U^c$ is identified with $U^\vee \otimes \rho_{\chi,\lambda}$ as representations of $G_F$. 
\end{proof}

\begin{proposition}\label{prop: conjugate self-duality criteria from single prime}
Let $\lambda,\lambda'$ be primes of $M_\pi$. Suppose that $c_{\lambda}$ acts as a trivial outer automorphism of $G_\lambda^{\circ,\der}$.
\begin{enumerate}
    \item Then $c_{\lambda'}$ acts trivially on $X^*(T_{\lambda'})$.
    \item Suppose, in addition, that $U$ is an irreducible $G_F$-subrepresentation of $\rho_{\pi,\lambda'}$ such that, when viewed as a representation of $\mf{g}_{\lambda'}$, contains a non-trivial irreducible self-dual subrepresentation.
Then $U$ is polarizable.
\end{enumerate}
\end{proposition}
\begin{proof}
Note that $c_{\lambda}$ acts trivially on $X^*(T_{\lambda})$, as is argued in the discussion following \cite[Proposition 21]{Xia19}.
By Proposition \ref{prop: complex conjugation formal character action}, we see that $c_{\lambda'}$ acts trivially on $X^*(T_{\lambda'})$.
For the second part, let $w$ be a non-zero weight of a non-trivial irreducible self-dual summand of $U$, on viewing $U$ as a representation of $\mf{g}_{\lambda'}$.
Then $-w$ is also a weight of $U$.
Then since $c_{\lambda'}$ acts trivially, we see that $-w$ is a weight of $U^c$ and the result follows from Lemma \ref{lemma: conjugate self dual summand criterion}.
\end{proof}

\begin{lemma} \label{lemma: summand has same rank}
Let $K/F$ and $M/M_\pi$ be finite extensions.
Let $(\sigma_\eta)_\eta$ be a weakly compatible system of Galois representations of $K$ defined over $M$.
Suppose that $0 \not\in W_{\lambda_0}$ and that there exists a prime $\eta$ of $M$ lying above a prime $\lambda$ of $M_\pi$ such that
$\sigma_\eta$ is a direct summand of $\rho_{\pi,\lambda}|_{G_K}$.
Then $\rk \sigma_\eta = \rk \mf{g}$ and $\dim \sigma_\eta$ is at least the smallest possible dimension of a multiplicity-free faithful irreducible representation of $\mf{g}$.
\end{lemma}
\begin{proof}
We may form a new weakly compatible system
\[
(\tau_\eta := \rho_{\pi,\eta}|_{G_{K}} \oplus \sigma_\eta)_\eta
\]
of $(n+m)$-dimensional Galois representations of $G_{K}$, where $m = \dim \sigma_\eta$.
Since $\sigma_\eta$ is a direct summand of $\rho_{\pi,\lambda}|_{G_K}$, we see that $\mf{g}_{\tau_\eta} \cong \mf{g}_{\rho_{\pi,\eta}}$ has rank equal to that of $\mf{g}$.
It follows from Theorem \ref{thm: formal char independent of l} that $\rk \mf{g}_{\tau_{\eta_0}} = \rk \mf{g}$, where $\eta_0| \lambda_0$ is a prime of $M$.
Since $\rho_{\pi,\eta_0}$ is strongly irreducible with $\mf{g}_{\rho_{\eta_0}} \cong \mf{g}$ simple, it follows that $\mf{g}_{\tau_{\eta_0}}$ admits $\mf{g}$ as a simple factor.
Since $\rk \mf{g}_{\tau_{\eta_0}} = \rk \mf{g}$, we deduce that $\mf{g}_{\tau_{\eta_0}} \cong \mf{g}$.
By a version of Goursat's lemma, $\mf{g}_{\sigma_{\eta_0}}$ must either be isomorphic to $\mf{g}$ or $0$.
We cannot have $\mf{g}_{\sigma_{\eta_0}} = 0$, else by Theorem \ref{thm: formal char independent of l} we would have $\mf{g}_{\sigma_{\eta}} = 0$ and $\rho_{\pi,\lambda}$ would admit $0$ as a weight, contradicting that $0 \not\in W_{\lambda}$.
Thus $\rk \sigma_\eta = \rk \sigma_{\eta_0} = \rk \mf{g}$.
Since $\sigma_{\eta_0}$ is a faithful representation of $\mf{g}_{\sigma_{\eta_0}} \cong \mf{g}$ and is multiplicity-free (because $\sigma_\eta$ is), the final part now follows.
\end{proof}

\subsection{Proofs of irreducibility} \label{subsect: simple type irreducibility proofs}

We are almost ready to begin proving each case of Theorem \ref{thm: reduction to simple tensorands}.
The following lemma will be useful in controlling the component group of $G_\rho$.

\begin{lemma} \label{lemma: pi_0 embeds in out(g)}
Let $\Omega$ be an algebraically closed field of characteristic zero and let $H \subset \GL_{n,\Omega}$ be a possibly disconnected, algebraic subgroup with $H^\circ$ acting irreducibly.
Let $Z = Z(\GL_n)$ and let $\mf{h}' = \Lie(H^{\circ,\der})$.
Then there are natural inclusions
\[Z H/(Z H^{\circ,\der}) \hookrightarrow N_{\GL_n}(Z H^{\circ,\der})/(ZH^\circ) \hookrightarrow \Out(\mf{h}').\]
\end{lemma}
\begin{proof}
Since $Z H$ normalizes $H^{\circ,\der}$, there is an inclusion
\[
Z H/ (Z H^{\circ,\der}) = Z H/ (Z H^\circ) \hookrightarrow N_{\GL_n}(Z H^{\circ,\der})/(ZH^\circ),
\]
with the first equality holding by Schur's lemma.
If $g \in N_{\GL_n}(H^{\circ,\der})$ acts as an inner automorphism of $H^{\circ,\der}$, then $g = hz$, for some $h \in H^{\circ,\der}$ and $z \in Z_{\GL_n}(H^{\circ,\der}) = Z$.
We deduce that there is an injection
\[
N_{\GL_n}(Z H^{\circ,\der})/(ZH^\circ) \hookrightarrow \Out(H^{\circ,\der}) \cong \Out(\mf{h}'). \qedhere
\]
\end{proof}

Due to the next lemma, it will suffice to only establish \emph{irreducibility} of $\rho_{\pi,\lambda}$ in a density one set, in order to prove Theorem \ref{thm: reduction to simple tensorands}.

\begin{lemma}
Let $\lambda$ be a prime of $M_\pi$ such that $\rho_{\pi,\lambda}$ is irreducible.
Then $\rho_{\pi,\lambda}$ is strongly irreducible.
\end{lemma}
\begin{proof}
After possibly replacing $\pi$ with its twist by an infinite order character, we may suppose that $Z(\GL_n) \subset G_{\rho}^\circ$, as can be seen by considering determinants.
We then claim that $F^\circ/F$ is a cyclic extension of degree at most $2$.
In the case that $\mf{g} \not\cong \solie_8$, the claim follows from combining Lemma \ref{lemma: outer automorphisms simple lie algebras} and Lemma \ref{lemma: pi_0 embeds in out(g)}.
If $\mf{g} \cong \solie_8$, then by Lemma \ref{lemma: formal char is multiplicity-free} and Theorem \ref{thm: multiplicity-free irreps} we see that $V$ is given by either the standard representation or one of the two half-spin representations.
Since these three representations are permuted under the action of the $\Out(\solie_8) \cong S_3$, it is straightforward to see that $N_{\GL_8}(G_{\rho}^\circ) \cong \GO_8$.
The claim now follows from Lemma \ref{lemma: pi_0 embeds in out(g)} again, since $\pi_0(\GO_8) \cong \mathbb{Z}/2\mathbb{Z}$.

Suppose that $\rho_{\pi,\lambda}$ is not strongly irreducible. Then $\rho_{\pi,\lambda}$ becomes reducible over $F^\circ$ and we therefore have (by Clifford theory and Frobenius reciprocity) an isomorphism
$\rho_{\pi,\lambda} \cong \rho_{\pi,\lambda} \otimes \delta_{F^\circ/F}$, where $\delta_{F^\circ/F}$ is the quadratic character cut out by $F^\circ/F$.
We deduce (by Chebotarev density, say) that $\rho_{\pi,\lambda_0} \cong \rho_{\pi,\lambda_0} \otimes \delta_{F^\circ/F}$.
This contradicts strong irreducibility of $\rho_{\pi,\lambda_0}$, by \cite[Lemma 4.8 (1)]{kret2022galois}.
\end{proof}

The following four propositions handle the cases that $\mf{g}$ is of type $A$.

\begin{proposition} \label{prop: sl_2 irred}
Suppose that $\mf{g} \cong \sllie_{2}$.
Then the conclusion to Theorem \ref{thm: reduction to simple tensorands} holds.
\end{proposition}
\begin{proof}
We know that $V$ is isomorphic to the representation $\Sym^{n-1} \Std$ of $\sllie_2$ of dimension $n$.
Let $\lambda | \ell$ be a prime of $M_\pi$.
By Theorem \ref{thm: formal char independent of l}, we must have $\mf{g}_\lambda \cong \sllie_2$, since it is the unique semisimple Lie algebra of rank $1$ (up to isomorphism).
It follows that every irreducible summand of $\rho_{\pi,\lambda}$ is self-dual as a representation of $\mf{g}_\lambda$.
Since $\sllie_2$ has no non-trivial outer automorphisms, Proposition \ref{prop: conjugate self-duality criteria from single prime} shows that every summand of $\rho_{\pi,\lambda}$ is either polarizable, or a character (by multiplicity-freeness) and therefore also polarizable.
The result in this case now follows from Theorem \ref{thm: irreducibility from polarized summands}.    
\end{proof}

\begin{proposition} \label{prop: sl_2k+1 irred}
Suppose that $\mf{g} \cong \sllie_{2k+1}$ with $k \geq 1$.
Then the conclusion to Theorem \ref{thm: reduction to simple tensorands} holds.
\end{proposition}
\begin{proof}
Observe that $V$ is not self-dual, by Theorem \ref{thm: multiplicity-free irreps}.
There exists an irreducible algebraic representation $\theta: \GL_{2k+1} \to \GL_n$ such that $\rho$ (up to conjugacy) factors through $H = \theta(\GL_{2k+1}) \subset \GL_n$.
Indeed, $G_{\rho}^\circ \subset G_{\rho}^{\circ,\der} Z(\GL_n)$ is contained inside such a subgroup $H$.
However, $\pi_0(G_\rho)$ cannot act as a non-trivial outer automorphism of $\mf{g}$, since $V$ is not self-dual; Lemma \ref{lemma: pi_0 embeds in out(g)} then implies that $G_{\rho} \subset G_{\rho}^\circ Z(\GL_n) \subset H$.

Consider the composition $\rho': G_F \to H(\Qlbar) \to \PGL_{2k+1}(\Qlbar)$.
We can find a geometric Galois representation $\tilde{\rho}: G_F \to \GL_{2k+1}(\Qlbar)$ lifting $\rho'$ by \cite[Theorem 3.2.10]{patrikis2019variations}.
On twisting $\tilde{\rho}$ by a character, we may suppose that $\tilde{\rho}$ is moreover crystalline (by applying \cite[Theorem A.0.5]{liu2016automorphy} to the crystalline representation $\tilde{\rho} \otimes \tilde{\rho}^\vee \cong \mathrm{ad}(\rho') \oplus \mathbf{1}$, together with global class field theory).
Since $(\rho')^{c} \cong (\rho')^\vee$, it follows that $\tilde{\rho}^c \cong \tilde{\rho}^\vee \otimes \varphi$ for some character $\varphi: G_F \to \Qlzerobar^\times$.
Since the dimension of the space of $G_F$-coinvariants of $\Tilde{\rho} \otimes \Tilde{\rho}^c$ can be at most one dimensional, we see that $\varphi$ extends to a character of $G_{F^+}$.
We see that $\tilde{\rho}$ is polarizable and totally odd by Lemma \ref{lemma: oddness}.
Moreover, $\theta \circ \tilde{\rho} \cong \rho \otimes \delta$ for some (crystalline) geometric character $\delta: G_F \to \Qlzerobar^\times$.
If $\tilde{\rho}$ was not of regular weight (resp. Fontaine--Laffaille), then since $\theta$ is given by (possibly the dual of) some symmetric  or exterior power, it would follow that $\rho$ was not of regular weight (resp. Fontaine--Laffaille), a contradiction to our assumptions on $\rho$.
Since $\theta(\ol{\tilde{\rho}|_{G_{F(\zeta_{\ell_0})}}})$ is irreducible, we see that $\ol{\tilde{\rho}|_{G_{F(\zeta_{\ell_0}})}}$ is also irreducible (as $\theta$ is isomorphic a non-trivial to a symmetric or alternating power of the standard representation, by multiplicity-freeness).
We can therefore apply Theorem \ref{thm: blggt potential automorphy} to $\tilde{\rho}$ to obtain finite extensions $K/F$ and $M/M_\pi$, and a weakly compatible system $(\sigma_\eta)_\eta$ of $G_K$ defined over $M$ such that 
\(
\tilde{\rho}_{\pi,\lambda_0}|_{G_K} \cong \sigma_{\eta_0}
\)
for some prime $\eta_0|\lambda_0$ of $M$.
Since $\delta: G_F \to \Qlzerobar^\times$ is geometric, it is also part of a weakly compatible system $(\delta_\eta)_\eta$ defined over $M$ (on possibly replacing $M$ by a further finite extension).

Let $\eta$ be a prime of $M$.
Let $\mf{h}_\eta = \mf{g}_{\sigma_\eta}$.
We claim that $\mf{h}_\eta \cong \sllie_{2k+1}$ acting irreducibly via the standard representation or its dual.
This is because $\rk \mf{h}_\eta = 2k$ by Theorem \ref{thm: formal char independent of l}, while the only faithful representations of a rank $2k$ semisimple Lie algebra of dimension $2k+1$  (up to isomorphism) are given by the standard representation of $\sllie_{2k+1}$ or its dual.

By the Chebotarev density theorem, we must have $\theta(\sigma_\eta) \cong \rho_{\pi,\lambda}|_{G_K} \otimes \delta_\eta$ for $\lambda$ the prime of $M_\pi$ above $\eta$, since, for all but finitely primes $v$ of $F$, the characteristic polynomials of the image of $\Frob_v$ under either representation coincide. 
Since $\theta$ induces an irreducible representation of $\mf{h}_\eta$, we see that $\rho_{\pi,\lambda}|_{G_K}$ is irreducible too, which concludes the proof.
\end{proof}

\begin{proposition}\label{prop: sl_2k non-self-dual irred}
Suppose that $\mf{g} \cong \sllie_{2k}$ for some $k \geq 2$ and some non-self-dual representation $V$.
Then the conclusion to Theorem \ref{thm: reduction to simple tensorands} holds.
\end{proposition}
\begin{proof}
We know that $c_{\lambda_0}$ must act as a non-trivial outer automorphism of $\mf{g}$.
It follows that $\dim X^*(T_{\lambda_0})^{c_{\lambda_0} = 1} = k$.
Let $\lambda| \ell$ be a prime.
Let $\mf{h}_\lambda$ be a maximal type $A$ subalgebra of $\mf{g}_\lambda$.
We claim that $\mf{h}_\lambda$ is isomorphic to $\mf{g}_\lambda$.

If $k \geq 5$, then by Lemma \ref{lemma: maximal type A subalgebras} and Theorem \ref{thm: formal char type a subalgebra equivalence} we see that $\mf{h}_\lambda \cong \sllie_{2k}$ and the claim follows.
If $k=3$, then 
Lemma \ref{lemma: type A_n isomorphic for most n} shows that $\mf{h}_\lambda \cong \sllie_6$.
If $k=2$ or $k = 4$, we firstly show that $V$ must be isomorphic to $\Sym^a \Std$ (or its dual), where $\Std$ is the standard representation of $\sllie_{2k}$ and $a \geq 1$.
By Theorem \ref{thm: multiplicity-free irreps}, the only other possible faithful irreducible multiplicity-free representations of $\sllie_{2k}$ are the alternating powers of the standard representation.
When $k=4$, the possible dimensions of these proper alternating powers are $28$, $56$ or $70$, each of which cannot occur by our assumption that $7 \nmid n$.
When $k=2$, the only possibility is $\Lambda^2 \Std$, which we can exclude because $V$ is not self-dual.
We therefore see that $\mf{h}_\lambda \cong \sllie_{2k}$ by Proposition \ref{prop: symmetric power type A are iso}, establishing the claim.

With the claim in hand note that, by considering maximal type $A$ subalgebras of simple Lie algebras obtained via \cite[Table 1]{hui2013monodromy} and Lemma \ref{lemma: maximal type A subalgebras}, in every case we must have $\mf{g}_\lambda$ itself is of type $A$ or $\mf{g}_\lambda \cong \mf{e}_7$.
The case of $\mf{g}_\lambda \cong \mf{e}_7$ can be ruled out because multiplicity-freeness, Theorem \ref{thm: multiplicity-free irreps} and the fact that $0$ is not a weight of $V$ (ruling out any $1$-dimensional summands) would together imply that $n = 56$, a multiple of $7$.
We therefore have that $\mf{g}_\lambda \cong \sllie_{2k}$,
with $c_\lambda$ acting as the unique non-trivial outer automorphism of $\mf{g}_\lambda$ by Proposition \ref{prop: complex conjugation formal character action}.
We deduce that every irreducible summand $U \leq V_\lambda$ satisfies $U^{c_\lambda} \cong U^\vee$.
It follows from Lemma \ref{lemma: conjugate self dual summand criterion} that each summand of $\rho_{\pi,\lambda}$ is polarizable and the result now follows from Theorem \ref{thm: irreducibility from polarized summands}.
\end{proof}
\begin{proposition}\label{prop: sl_2k self-dual irred}
Suppose that $\mf{g} \cong \sllie_{2k}$ for some $k \geq 3$ with $V$ self-dual.
Then the conclusion to Theorem \ref{thm: reduction to simple tensorands} holds.
\end{proposition}
\begin{proof}
As a representation of $\sllie_{2k}$, we know from Theorem \ref{thm: multiplicity-free irreps} that $V \cong \Lambda^k \Std$ with $n = {{2k}\choose{k}}$.
We claim that $\mf{g}_\lambda$ contains a maximal rank type $A$ subalgebra $\mf{h}_\lambda$ isomorphic to $\sllie_{2k}$.
If $k \geq 5$, then the claim follows from Lemma \ref{lemma: maximal type A subalgebras} and Theorem \ref{thm: formal char type a subalgebra equivalence}.
If $k=4$ then $n = 70$, so this case does not occur by our assumption that $7 \nmid n$.
If $k=3$, then Lemma \ref{lemma: type A_n isomorphic for most n} shows that $\mf{h}_\lambda = \mf{g}_\lambda \cong \sllie_6$.

Recall from \S \ref{subsect: type A formal characters} that there is a natural inner product $\langle-,-\rangle_{\lambda_0}$ (resp. $\langle-,-\rangle_{\lambda}$) on the weight space of $\mf{g}$ (resp. $\mf{h}_\lambda$).
As mentioned there, these inner product spaces are isometric under the isomorphism of formal characters given by Theorem \ref{thm: formal char independent of l}.
Since every weight $w \in W_{\lambda_0}$ has the same norm with respect to this inner product (Lemma \ref{lemma: alternating power inner products}), the same is therefore true of weights in $W_\lambda$.

Suppose that the representation $U$ of $\mf{h}_\lambda$ on $\Qlbar^n$ is reducible.
As a representation of $\sllie_{2k}$, $U$ is isomorphic to a direct sum of (distinct) alternating or symmetric powers of the standard representation, by Theorem \ref{thm: multiplicity-free irreps}.
The symmetric powers $\Sym^a \Std$ (or their duals) for $a \geq 2$ do not appear as a constituent of $U$.
This is because any such symmetric power admits weights with distinct norms by Lemma \ref{lemma: symmetric power weights}.
Furthermore, we must have $U \cong \Lambda^a \Std \oplus \Lambda^{2k-a} \Std$ for some $1 \leq a \leq k-1$.
Indeed, Lemma \ref{lemma: alternating power inner products} states that (up to a fixed scalar) the squared norms of the weights of the representation $\Lambda^a \Std$ are all given by $a(2k-a)$. 
The expression $a(2k-a)$ is quadratic in $a$, attaining its maximum at $a = k$, and can therefore only take on the repeated values at $a$ and $2k-a$.
We see from Lemma \ref{lemma: alternating power inner products} that
\begin{align*}
 1 - \frac{2}{k} &= \frac{\max_{v \in W_{\lambda_0} \setminus \{w\}} \langle w, v \rangle_{\lambda_0}}{||w||_{\lambda_0}^2} \\
 &= \frac{\max_{v' \in W_{\lambda} \setminus \{w'\}} \langle w', v' \rangle_{\lambda}}{||w'||_{\lambda}^2} \\
 &= \max(1- \frac{2k}{a(2k-a)},\frac{2k}{2k-a} - 1).
\end{align*}
Similarly to the above, we cannot have $\frac{2}{k} = \frac{2k}{a(2k-a)}$, since $a < k$.
If $1- \frac{2}{k} = \frac{2k}{2k-a} - 1$ then $a = k-1 - \frac{1}{k-1}$ cannot be an integer, since $k \geq 5$.
In either case we obtain a contradiction, from which we see that $\mf{h}_\lambda$, and hence $\mf{g}_\lambda$, must  act irreducibly.
The result now follows.
\end{proof}

\begin{remark}\label{rmk: hui's comment}
    Hui brought to our attention that a significantly shorter proof of \Cref{prop: sl_2k self-dual irred} is possible using a modified version of \cite[Theorem 3.10]{Hui18}; while the theorem assumes that each member of our compatible system is defined over \( \bb{Q}_l \) (coming from the étale cohomology of a smooth projective variety), these assumptions are not actually necessary, as the proof only relies on the existence of an isomorphism of formal characters after base-change to $\bb{C}$.
    Proposition \ref{prop: sl_2k self-dual irred} for $k \geq 5$ immediately follows from this, and the remaining cases of $k=2,3$ can be excluded in an identical manner to the proof of Proposition \ref{prop: sl_2k self-dual irred}. Although this observation may also be applied to \Cref{prop: sl_2k+1 irred} and \Cref{prop: sl_2k non-self-dual irred} for large values of \( k \), we do not know how to give a proof for small \( k \) without using the arguments that we have given (which essentially work for all \( k \)).
\end{remark}

We now consider the cases of the other classical simple Lie algebras, where the following lemma will be useful in finding polarizable summands.

\begin{lemma} \label{lemma: so_q possible formal characters}
Let $\mathfrak{h}$ be a semisimple Lie algebra and $W$ a faithful representation of $\mf{h}$.
Suppose that $W$ has the same formal character as the standard $m$-dimensional representation of $\solie_{m}$, $m \geq 3$.
Then every irreducible summand of $W$ is self-dual.
\end{lemma}
\begin{proof}
Without loss of generality, we may suppose that $\mf{h}$ is of type $A$.
To see this, firstly note that $\mf{h}$ contains a type $A$ subalgebra $\mf{h}' \subset \mf{h}$ of equal rank by Lemma~\ref{lemma: maximal type A subalgebras} and that the restriction $W|_{\mf{h}'}$ has the same formal character as $\mf{h}$.
Moreover, if every irreducible summand of $W|_{\mf{h}'}$ is self-dual then the same conclusion holds for $W$, since each irreducible summand will of $W$ be a sum of irreducible summands of $W|_{\mf{h}'}$ and self-duality can be checked on the level of characters.

Write $\mf{h} \cong \prod_{i=1}^k \sllie_{m_i+1}$ for some integers $m_i \geq 1$.
Suppose that (after reordering) $W$ contains an irreducible, non-self-dual summand of the form $U \cong (\boxtimes_{i=1}^j U_i) \boxtimes (\boxtimes_{i=j+1}^k \mathbf{1})$ for some $1 \leq j \leq k$ and faithful irreducible representations $U_i$ of $\sllie_{m_i}$.
Let $X \cong U \oplus U^\vee$ be the minimal self-dual summand of $W$ containing $U$.
Then $\dim X \leq \dim W - 2 \sum_{i=j+1}^k m_i$.
This is because the complement of $X$ inside $W$ must restrict to a faithful self-dual representation of $\prod_{i=j+1}^k \sllie_{m_i+1}$, while the dimension of any faithful self-dual representation of a semisimple Lie algebra is always at least twice the rank.
Since $\dim W \leq 1 + 2\sum_{i=1}^k m_i$ (because $\solie_m$ has the same rank as $\mf{h}$), we have
\[
2 \dim U = \dim X \leq 1 + 2\sum_{i=1}^j m_i.
\]
Since each $U_i$ is a faithful irreducible representation of $\sllie_{m_i}$, we have $\dim U_i \geq 1+m_i$.
We obtain a contradiction, since we have inequalities
\[
1+ \sum_{i=1}^j m_i \leq \prod_{i=1}^j (1+m_i) \leq \dim U \leq \frac{1}{2} + \sum_{i=1}^j m_i. \qedhere
\] 
\end{proof}

\begin{proposition}\label{prop: types B, C, D irred}
Suppose that $\mf{g}$ is of type $B$, $C$ or $D$ with $\dim V \leq 2 \rk \mf{g} + 1$.
Then the conclusion to Theorem \ref{thm: reduction to simple tensorands} holds.
\end{proposition}
\begin{proof}
If $\mf{g}$ is of rank $1$, this case has already been covered in Proposition \ref{prop: sl_2 irred}.
Suppose that $\mf{g} \cong \solie_{n}$ for $n \geq 5$.
Then by Theorem \ref{thm: multiplicity-free irreps}, we may suppose that $V$ is given by the standard representation (noting in the case that $\mf{g} \cong \solie_8$, we may change this isomorphism by an outer automorphism of $\solie_8$ so that this is still the case; see \cite[20.3]{fulton2013representation}).
Let $\lambda|\ell$ be a prime of $M_\pi$.
We see from Lemma \ref{lemma: so_q possible formal characters} and Theorem \ref{thm: formal char independent of l} that every irreducible summand of the representation of $\mf{g}_\lambda$ on $\Qlbar^n$ is self-dual.
If $c_{\lambda_0}$ acts as an inner automorphism, Proposition \ref{prop: conjugate self-duality criteria from single prime} shows that every summand of $\rho_{\pi,\lambda}$ is either polarizable, or a character (by multiplicity-freeness) and therefore also polarizable.
The result in this case follows from Theorem \ref{thm: irreducibility from polarized summands}.
Otherwise, by Lemma \ref{lemma: outer automorphisms simple lie algebras}, $n$ must be even and $c_{\lambda_0}$ must act via the order $2$ outer automorphism of $\solie_n$ arising from the action of the orthogonal group (in order for the standard representation to be preserved under its action). 
In this case, one can check (cf. \cite[Exercise 19.9]{fulton2013representation}) that $W_{c_{\lambda_0}}$ contains $0$ with multiplicity $2$.

It follows from Proposition \ref{prop: complex conjugation formal character action} that there exists an irreducible summand $U \leq \rho_{\pi,\lambda}$ which, as a representation of $\mf{g}_\lambda$, admits a weight $w$ for which $-w$ is a weight of $U^{c_\lambda}$.
Suppose that $\lambda|\ell$ with $\ell \in \mathcal{L}_\pi$.
Then by Proposition \ref{proposition: density 1 residual irreducibility of summands} and Theorem \ref{thm: blggt potential automorphy} there exists a finite CM extension $F'/F$ and a regular algebraic cuspidal polarized automorphic representation $\pi'$ of $\GL_m(\bb{A}_{F'})$ whose associated $\ell$-adic Galois representation is isomorphic to $U|_{G_{F'}}$, where $m = \dim U$.
Let $M/M_\pi$ be a finite extension such we can attach a compatible system $(\sigma_\eta)_\eta$ of Galois representations to $\pi'$ defined over $M$.
We see from Lemma \ref{lemma: summand has same rank} that if $n >6$, then $m \geq n$, since $n$ is then the minimal dimension of a multiplicity-free faithful representation of $\solie_n$ by Theorem \ref{thm: multiplicity-free irreps}.
If $n=6$ and $m < n$, then Lemma \ref{lemma: summand has same rank} implies that $m \geq 4$ with $U$ non-self-dual.
Self-duality of $V_{\lambda}$ would force $V_\lambda$ to contain both $U$ and $U^\vee$, which gives a contradiction as $n=6$.
Consequently, $m=n$ and $\rho_{\pi,\lambda}$ is irreducible.

Suppose that $n \geq 4$ is even and $\mf{g} \cong \splie_n$.
Then $V$ isomorphic to the standard representation, by Theorem \ref{thm: multiplicity-free irreps}.
Then the formal character of $V$ is the same as that of the standard representation of $\solie_n$ and $\Out(\splie_n)$ is trivial.
The result now follows by the same argument as in the case of $\solie_n$.    
\end{proof}

There is one remaining case where $\mf{g}$ is a classical Lie algebra, which is when $\mf{g} \cong \solie_7$ with $V$ given by the spin representation. 

\begin{lemma} \label{lemma: gspin regular cocharacter}
Let $\alpha: \mathbb{G}_m \to \GSpin_7$ be a cocharacter.
Let $\beta$ (resp. $\gamma$) denote the composition $\Spin \circ \alpha: \mathbb{G}_m \to \GSpin_7 \xrightarrow{\Spin} \GL_8$ (resp. $\Std \circ \alpha: \mathbb{G}_m \to \GSpin_7 \xrightarrow{\Std} \GL_7$).
Viewing $\beta$ (resp. $\gamma$) as a representation of $\mathbb{G}_m$, we obtain integers $b_1,\ldots,b_8$ (resp. $c_1,\ldots,c_7$) corresponding to the weights of the representation.
\begin{enumerate}
    \item If the multiset $\{b_1,\ldots,b_8\}$ contains no repeated elements, then the multiset $\{c_1,\ldots,c_7\}$ contains no repeated elements.
    \item The inequality $\max_{i,j} |c_i-c_j| \leq 2 \max_{i,j} |b_i-b_j|$ always holds.
\end{enumerate}
\end{lemma}
\begin{proof}
From the weights of the $\Spin$ and $\Std$ representations, we have equalities of multisets 
\begin{align*}
  \{b_i\} &= \{d + \sum_{i=1}^3 \eps_i a_i : (\eps_i)_i \in \{\pm 1\}^3 \} \\
  \{c_i\} &= \{2d \} \cup \{2d + 2\eps_i a_i : \eps_i \in \{\pm 1\}, 1 \leq i \leq 3\}
\end{align*}
for some integers $a_i$ and $d$.
For $\gamma$ to admit a repeated weight, we must have either some $a_i = 0$ or $a_i = \pm a_j$ for some $i \neq j$.
In either case, $\beta$ would admit a repeated weight, from which the first part now follows.
For the second part, note that 
\[
\max_{i,j} |c_i-c_j| = 2 \max_{i}  |a_i| \leq 2 \sum_{i=1}^3 |a_i| = 2 \max_{i,j}  |b_i-b_j|. \qedhere
\]
\end{proof}

\begin{proposition}\label{prop: so_7 spin rep}
Suppose that $\mf{g}$ is of type $B$, $C$ or $D$ with $\dim V > 2 \rk \mf{g} + 1$.
Then the conclusion to Theorem \ref{thm: reduction to simple tensorands} holds.
\end{proposition}
\begin{proof}
We see from Lemma \ref{lemma: formal char is multiplicity-free}, Theorem \ref{thm: multiplicity-free irreps} and the assumptions that $7 \nmid n$ and $16 \nmid n$ that $\mf{g} \cong \solie_7$ with $V$ isomorphic to the $8$-dimensional spin representation.
We claim that the image of $\rho$ is a subgroup of $\GSpin_7(\Qlzerobar)$.
To see this, note that the outer automorphism group of $\Spin_7$ is trivial by Lemma \ref{lemma: outer automorphisms simple lie algebras}.
By Lemma \ref{lemma: pi_0 embeds in out(g)}, we see that the normalizer of $\GSpin_7$ inside $\GL_8$ is equal to $\GSpin_7$, from which the claim follows.

It is easy to see (from highest weight considerations) that the tensor square of the spin representation of $\solie_7$ decomposes as a direct sum of four irreducible representations, corresponding to the $i$th exterior powers of the standard representation for $0 \leq i \leq 3$.
It follows that there exists a unique $7$-dimensional irreducible summand $\sigma \subset \rho \otimes \rho$.
Since $(\rho \otimes \rho)^c \cong (\rho \otimes \rho)^\vee \otimes \chi|_{G_F}^2$ for some character $\chi$ of $G_{F^+}$, we see from uniqueness of $\sigma$ that $\sigma^c \cong \sigma^\vee \otimes \chi|_{G_F}^2$.
Thus $\sigma$ is polarizable and moreover totally odd by Lemma \ref{lemma: oddness}.

The representation $\sigma$ is crystalline, since $\sigma$ is a summand of the crystalline representation $\rho \otimes \rho$.
We additionally claim that $\sigma$ has regular Hodge--Tate weights in the Fontaine--Laffaille range.
Indeed, if $\mu_\tau: (\mathbb{G}_m)_C \to (\GL_8)_C$ is the Hodge--Tate cocharacter of $\rho$ labelled by some embedding $\tau: \Qlzerobar \hookrightarrow C := \widehat{\Qlzerobar}$ (as in \cite[2.4]{buzzardgee}), then $\mu_\tau$ is valued in $(\GSpin_7)_C$, thought of as a subgroup via (a twist of) the spin representation.
It is straightforward to see from the construction of $\sigma$ that the Hodge--Tate cocharacter of $\sigma$ attached to $\tau: \Qlzerobar \hookrightarrow C$ arises (up to twist) from the composition $(\mathbb{G}_m)_C \to (\GSpin_7)_C \xrightarrow{\Std} (\GL_7)_C$.
The claim now follows from Lemma \ref{lemma: gspin regular cocharacter} together with our assumption that $\HT_\tau(\rho) \subset [a_\tau, a_\tau + 2 \ell_0 - 4]$ for every $\tau: F \hookrightarrow \Qlzerobar$.

We next claim that $\overline{\sigma}(G_{F(\zeta_{\ell_0})})$ is irreducible.
Indeed, the irreducible representation $\overline{\rho}|_{G_{F(\zeta_{\ell_0})}}$ factors as $G_{F(\zeta_{\ell_0})} \to \GSpin_7(\overline{\mathbb{F}_{\ell_0}}) \hookrightarrow \GL_8(\overline{\mathbb{F}_{\ell_0}})$.
This shows that the subgroup $\tilde{\Gamma} := \rho(G_{F(\zeta_{\ell_0})}) \subset \GSpin_7(\overline{\mathbb{F}_{\ell_0}})$ is $\GSpin_7$-irreducible, in the sense that $\tilde{\Gamma}$ does not lie in any proper parabolic subgroup of $\GSpin_7$.
Let $\Gamma \subset \GO_7(\overline{\mathbb{F}_{\ell_0}})$ denote the image of $\tilde{\Gamma} \subset \GSpin_7(\overline{\mathbb{F}_{\ell_0}})$ under the standard representation of $\GSpin_7$, a $\GO_7$-irreducible subgroup due to the natural bijection between parabolic subgroups of $\GSpin_7$ and $\GO_7$. 
Thus, if $\Gamma$ was contained in a proper parabolic subgroup of $\GL_7(\overline{\mathbb{F}_{\ell_0}})$, then $\Gamma$ must stabilize a proper subspace $W$ on which the quadratic form \( f \) defining $\GO_7$ is non-degenerate. It follows that \( f|_W \) is a non-degenerate quadratic form on \( W \). Write \( V=W\oplus W^\perp \), so that $\operatorname{Stab}_{\GSpin_7}(W) \cong (\GSpin(W) \times \GSpin(W^\perp))/\mu_2$ contains \( \wt{\Gamma} \).
The restriction of the spin representation of $\GSpin_7$ to $\operatorname{Stab}_{\GSpin_7}(W)$ is reducible.
Indeed, it is isomorphic to the external tensor product of the spin representations of $\GSpin(W)$ and $\GSpin(W^\perp)$, one of which must be of type $D$ and whose corresponding spin representation decomposes as a direct sum of two half-spin representations.
Since $\tilde{\Gamma}$ acts irreducibly, it follows that $\Gamma$ is an irreducible subgroup of $\GL_7(\overline{\mathbb{F}_{\ell_0}})$.
Since the projective image of $\Gamma$ is conjugate to the projective image of $\overline{\sigma}(G_{F(\zeta_{\ell_0})})$ in $\PGL_7(\overline{\mathbb{F}_{\ell_0}})$, the claim follows.

By Theorem \ref{thm: blggt potential automorphy}, we can find a finite CM extension $K/F$ and an automorphic representation $\pi'$ such that $\sigma|_{G_K} \cong \eta_{\lambda_0'}$, where $(\eta_\lambda)_\lambda$ is the weakly compatible system attached to $\pi'$, defined over some finite extension $M/M_\pi$ and $\lambda_0'|\lambda_0$ is some prime in $M$.
By the same argument as in Proposition \ref{prop: types B, C, D irred}, we can find a density one set of primes $\mathcal{L}_{\pi'}$ such that for every $\ell \in \mathcal{L}_{\pi'}$ and prime $\lambda|\ell$ of $M$, $\eta_\lambda$ is strongly irreducible.

Consider the weakly compatible system $(\rho_{\pi,\lambda}|_{G_K} \oplus \eta_\lambda)_\lambda$.
The derived subalgebra of the Lie algebra of $\rho_{\pi,\lambda_0}|_{G_K} \oplus \eta_{\lambda_0'} \cong (\rho_{\pi,\lambda_0} \oplus \sigma_{\pi,\lambda_0})|_{G_K}$ has rank $3$.
If $\lambda$ is a prime of $M$, then the derived subalgebra of the Lie algebra of $\rho_{\pi,\lambda_0}|_{G_K} \oplus \eta_{\lambda_0'}$ is also of rank $3$ by Theorem \ref{thm: formal char independent of l}.
If additionally $\lambda|\ell$ with $\ell \in \mathcal{L}_{\pi'}$, then since $\eta_\lambda$ is a strongly irreducible representation with rank $3$ derived image, the derived subalgebra of the Lie algebra of the image must be isomorphic to $\solie_7$ acting via the standard representation by Theorem \ref{thm: multiplicity-free irreps}.
It follows that $\mf{g}_\lambda \cong \solie_7$.
Since $0$ does not occur as a weight of $W_{\lambda_0}$ (nor therefore $W_\lambda$), we deduce that $V_\lambda$ is given by the spin representation of $\mf{g}_\lambda \cong \solie_7$ and $\rho_{\pi,\lambda}$ is irreducible.
\end{proof}

Using Theorem \ref{thm: multiplicity-free irreps}, we will show that the classical simple Lie algebras together with $\mf{e}_6$ cover all possibilities for $\mf{g}$, due to our assumptions on $n$.

\begin{proposition} \label{prop: e_6 irred}
Suppose that $\mf{g} \cong \mf{e}_6$.
Then the conclusion to Theorem \ref{thm: reduction to simple tensorands} holds.
\end{proposition}
\begin{proof}
By Theorem \ref{thm: multiplicity-free irreps}, $V$ is isomorphic to one of the non-self-dual $27$-dimensional irreducible representations of $\mf{e}_6$.
Since $V^{c_{\lambda_0}} \cong V^\vee$, we see that 
$\{-c_{\lambda_0}(w): w \in~W_{\lambda_0}\} = W_{\lambda_0} $.
We can therefore consider the involution 
\begin{align*}
s: W_{\lambda_0} &\to W_{\lambda_0} \\
w &\mapsto -c_{\lambda_0}(w)
\end{align*}
on the set $W_{\lambda_0}$ of size $27$.
Since $27$ is odd, this involution must have a fixed point.
It follows that $0$ occurs with non-zero multiplicity in $W_{c_{\lambda_0}}$.

Now let $\lambda| \ell$ for $\ell \in \mathcal{L}_\pi$.
Note that $V$, and hence $V_\lambda$, do not admit $0$ as a weight (\cite[\S 4.1]{boxer2019e6}).
Therefore, by Lemma \ref{lemma: conjugate self dual summand criterion} and Proposition \ref{prop: complex conjugation formal character action} we can find an irreducible, polarizable summand $U \leq \rho_{\pi,\lambda}$.
We see from Theorem \ref{thm: irreducibility from polarized summands} that $U$ is potentially automorphic.
By Lemma \ref{lemma: summand has same rank}, we see that $\dim U \geq 27 = \dim \rho_{\pi,\lambda}$, since $27$ is the smallest dimension of any faithful irreducible representation of $\mf{g}$.
It follows that $U = \rho_{\pi,\lambda}$ is irreducible.
\end{proof}

\begin{proof}[Proof of Theorem \ref{thm: reduction to simple tensorands}]
If $\mf{g}$ is of type $A$, the result follows from combining Propositions \ref{prop: sl_2 irred}, \ref{prop: sl_2k+1 irred}, \ref{prop: sl_2k non-self-dual irred} and \ref{prop: sl_2k self-dual irred}.
If $\mf{g}$ is of type $B$, $C$ or $D$, the result follows from Proposition \ref{prop: types B, C, D irred} and Proposition \ref{prop: so_7 spin rep}.
If $\mf{g}$ is of type $E_6$, the result follows from Proposition \ref{prop: e_6 irred}.
We see from Lemma \ref{lemma: formal char is multiplicity-free}, Theorem \ref{thm: multiplicity-free irreps} and the assumption that $7 \nmid n$, that there are no other possibilities for $\mf{g}$ and $V$, which concludes the proof.
\end{proof}

With Theorem \ref{thm: reduction to simple tensorands} (and hence Theorem \ref{thm: intro main theorem}) now established, we can deduce an irreducibility result for residual Galois representations.
In what follows, we drop our assumption that we are in the setup of Theorem \ref{thm: reduction to simple tensorands}, instead placing ourselves in the setup of Theorem \ref{thm: intro main theorem}.

\begin{proof}[Proof of Corollary \ref{cor: intro residual irreducibility}]
Let $\mc{L}$ (resp. $\mc{L}_\pi$) denote the density one set of rational primes obtained by applying Theorem \ref{thm: intro main theorem} (resp. Proposition \ref{proposition: density 1 residual irreducibility of summands}) to $(\pi,\chi)$.
Let $\ol{\mc{L}} = \mc{L} \cap \mc{L}_\pi$, a density one set of rational primes, and let $\ell \in \ol{\mc{L}}$ and $\lambda|\ell$.
As $\ell \in \mc{L}$, the representation $\rho_{\pi,\lambda}$ is irreducible.
Since in addition $\ell \in \mc{L}_\pi$, the representation $\ol{\rho_{\pi,\lambda}}|_{G_{F(\zeta_\ell)}}$ is therefore irreducible, which concludes the proof.
\end{proof}

\subsection{An example to illustrate the limitations of our methods}

As noted in the introduction, the assumption that $7 \nmid n$ is made because we cannot handle the case where the monodromy group at primes of irreducibility is given by $G_2$ with its $7$-dimensional representation.
The following example illustrates why we also need to impose some constraint on the power of $2$ dividing $n$.

\begin{example} \label{example: n=32 fails}
Suppose we have \( \pi \) in the setup of Theorem \ref{theorem: xia's reduction}, but without any assumption on $n$.
For a prime $\lambda$ of $M_\pi$, define $\mf{g}_\lambda$, $V_\lambda$ and $c_\lambda$ analogously to Section \ref{subsect: formal chars and complex conj}.
Let $\mc{L}$ be the positive Dirichlet density set of primes of \Cref{lemma: lie-irreducible patrikis taylor prime}. Suppose that for $\lambda_0|\ell_0$ with $\ell_0 \in \mathcal{L}$, $\rho_{\pi,\lambda_0}: G_F \to \GL_{32}(\Qlzerobar)$ satisfies $\mf{g}_{\lambda_0} \cong \splie_4 \times \solie_7$ under which $V_{\lambda_0} \cong \Std \boxtimes \Spin$.

We cannot show that the corresponding tensorands of $\rho_{\pi,\lambda_0}$ (cf. \Cref{prop: tensor decomposition over CM field}) are potentially automorphic because we cannot guarantee that they are both totally odd.
The representation of $\mf{g}_{\lambda_0}$ has the same formal character as $\solie_{10}$ acting via the direct sum of the two non-self-dual half-spin representations, $\Spin^+ \oplus \Spin^-$, because the latter has the same formal character as the spin representation of $\solie_{11}$, which in turn has the same formal character as the external tensor product of the spin representations of $\solie_4 \times \solie_7$.
Therefore, for \( \lambda\mid \ell \) in a positive Dirichlet density set, we cannot exclude the possibility that $\mf{g}_\lambda \cong \solie_{10}$ with $\rho_{\pi,\lambda}$ a direct sum of two non-polarizable summands. 
This is because neither summand would be self-dual up to twist and $c_\lambda$ must act trivially ($\Out(\splie_4 \times \solie_7)=1$ implies $c_{\lambda_0}$ acts trivially).

Thus, we cannot apply the potential automorphy results of \cite{BLGGT14} to any irreducible summand of $\rho_{\pi,\lambda}$.
Moreover, no functorial transfer arising from an irreducible representation of $\mf{g}_\lambda \cong \solie_{10}$ (cf. \Cref{prop: sl_2k+1 irred} and \Cref{prop: so_7 spin rep}) would be helpful, because the standard representation of $\solie_{10}$ is even dimensional, and we therefore cannot guarantee that the corresponding transfer would be a totally odd representation in order to apply potential automorphy.
\end{example}

\bibliographystyle{alpha}
\bibliography{odd}

\newcommand{\etalchar}[1]{$^{#1}$}
\begin{thebibliography}{BLGGT14}

\bibitem[BCE{\etalchar{+}}19]{boxer2019e6}
George Boxer, Frank Calegari, Matthew Emerton, Brandon Levin, Keerthi~Madapusi Pera, and Stefan Patrikis.
\newblock Compatible systems of {G}alois representations associated to the exceptional group ${E}_6$.
\newblock In {\em Forum of Mathematics, Sigma}, volume~7, page~e4. Cambridge University Press, 2019.

\bibitem[BG14]{buzzardgee}
Kevin Buzzard and Toby Gee.
\newblock The conjectural connections between automorphic representations and galois representations.
\newblock {\em Automorphic forms and Galois representations}, 1:135--187, 2014.

\bibitem[BH24]{BocHui24}
Gebhard Böckle and Chun-Yin Hui.
\newblock Weak abelian direct summands and irreducibility of {G}alois representations, 2024.

\bibitem[BLGGT14]{BLGGT14}
Thomas Barnet-Lamb, Toby Gee, David Geraghty, and Richard Taylor.
\newblock Potential automorphy and change of weight.
\newblock {\em Ann. of Math. (2)}, 179(2):501--609, 2014.

\bibitem[BR92]{BlaRog92}
Don Blasius and Jonathan~D. Rogawski.
\newblock Tate classes and arithmetic quotients of the two-ball.
\newblock In {\em The zeta functions of {P}icard modular surfaces}, pages 421--444. Univ. Montr\'eal, Montreal, QC, 1992.

\bibitem[CEG22]{CEG22}
Frank Calegari, Matthew Emerton, and Toby Gee.
\newblock Globally realizable components of local deformation rings.
\newblock {\em Journal of the Institute of Mathematics of Jussieu}, 21(2):533--602, 2022.

\bibitem[CG13]{CalGee13}
Frank Calegari and Toby Gee.
\newblock Irreducibility of automorphic {G}alois representations of {$GL(n)$}, {$n$} at most 5.
\newblock {\em Ann. Inst. Fourier (Grenoble)}, 63(5):1881--1912, 2013.

\bibitem[Con11]{Conrad11}
Brian Conrad.
\newblock Lifting global representations with local properties, 2011.
\newblock URL: \url{https://math.stanford.edu/~conrad/papers/locchar.pdf}. Last edited on 2011/12/11.

\bibitem[Far06]{fargues2006motives}
Laurent Fargues.
\newblock Motives and automorphic forms: the (potentially) abelian case.
\newblock {\em Preprint on webpage at \url{https://webusers.imj-prg.fr/~laurent.fargues/Motifs_abeliens.pdf}}, 2006.

\bibitem[FH13]{fulton2013representation}
William Fulton and Joe Harris.
\newblock {\em Representation theory: a first course}, volume 129.
\newblock Springer Science \& Business Media, 2013.

\bibitem[HLTT16]{HLTT16}
Michael Harris, Kai-Wen Lan, Richard Taylor, and Jack Thorne.
\newblock On the rigid cohomology of certain {S}himura varieties.
\newblock {\em Res. Math. Sci.}, 3:Paper No. 37, 308, 2016.

\bibitem[HT01]{harristaylorLLC}
Michael Harris and Richard Taylor.
\newblock {\em The geometry and cohomology of some simple {S}himura varieties}, volume 151 of {\em Annals of Mathematics Studies}.
\newblock Princeton University Press, Princeton, NJ, 2001.
\newblock With an appendix by Vladimir G. Berkovich.

\bibitem[Hui13]{hui2013monodromy}
Chun~Yin Hui.
\newblock {\em Monodromy of Galois representations and equal-rank subalgebra equivalence}.
\newblock PhD thesis, Indiana University, 2013.

\bibitem[Hui18]{Hui18}
Chun~Yin Hui.
\newblock On the rationality of certain type {A} {G}alois representations.
\newblock {\em Trans. Amer. Math. Soc.}, 370(9):6771--6794, 2018.

\bibitem[Hui23]{Hui23}
Chun~Yin Hui.
\newblock Monodromy of subrepresentations and irreducibility of low degree automorphic {G}alois representations.
\newblock {\em J. Lond. Math. Soc. (2)}, 108(6):2436--2490, 2023.

\bibitem[Hum12]{humphreys2012introduction}
James~E Humphreys.
\newblock {\em Introduction to {L}ie algebras and representation theory}, volume~9.
\newblock Springer Science \& Business Media, 2012.

\bibitem[KS22]{kret2022galois}
Arno Kret and Sug~Woo Shin.
\newblock Galois representations for general symplectic groups.
\newblock {\em Journal of the European Mathematical Society}, 25(1):75--152, 2022.

\bibitem[LP92]{LarPin92}
M.~Larsen and R.~Pink.
\newblock On {$l$}-independence of algebraic monodromy groups in compatible systems of representations.
\newblock {\em Invent. Math.}, 107(3):603--636, 1992.

\bibitem[LY16]{liu2016automorphy}
Tong Liu and Jiu-Kang Yu.
\newblock On automorphy of certain galois representations of go4-type.
\newblock {\em Journal of Number Theory}, 161:49--74, 2016.

\bibitem[Pat19]{patrikis2019variations}
Stefan Patrikis.
\newblock {\em Variations on a theorem of {T}ate}, volume 258.
\newblock American Mathematical Society, 2019.

\bibitem[PT15]{PatTay15}
Stefan Patrikis and Richard Taylor.
\newblock Automorphy and irreducibility of some {$l$}-adic representations.
\newblock {\em Compos. Math.}, 151(2):207--229, 2015.

\bibitem[Rib77]{Ribet77}
Kenneth~A. Ribet.
\newblock Galois representations attached to eigenforms with {N}ebentypus.
\newblock In {\em Modular functions of one variable, {V} ({P}roc. {S}econd {I}nternat. {C}onf., {U}niv. {B}onn, {B}onn, 1976)}, volume Vol. 601 of {\em Lecture Notes in Math.}, pages 17--51. Springer, Berlin-New York, 1977.

\bibitem[Sch15]{Scholz15}
Peter Scholze.
\newblock On torsion in the cohomology of locally symmetric varieties.
\newblock {\em Ann. of Math. (2)}, 182(3):945--1066, 2015.

\bibitem[SZ24]{sun2024descriptions}
Bin-Ni Sun and Yufeng Zhao.
\newblock Descriptions of strongly multiplicity free representations for simple {L}ie algebras.
\newblock {\em Journal of Algebra}, 644:655--689, 2024.

\bibitem[Tay95]{Taylor95}
Richard Taylor.
\newblock On {G}alois representations associated to {H}ilbert modular forms. {II}.
\newblock In {\em Elliptic curves, modular forms, \& {F}ermat's last theorem ({H}ong {K}ong, 1993)}, volume~I of {\em Ser. Number Theory}, pages 185--191. Int. Press, Cambridge, MA, 1995.

\bibitem[Xia19]{Xia19}
Yuhou Xia.
\newblock Irreducibility of automorphic {G}alois representations of low dimensions.
\newblock {\em Math. Ann.}, 374(3-4):1953--1986, 2019.

\end{thebibliography}

\end{document}